\documentclass[10pt]{amsart}
\pdfoutput=1 
\usepackage{appendix}
\usepackage{geometry}
\usepackage{amssymb}
\usepackage{amsmath}
\usepackage{amsthm}
\usepackage{mathtools}
\usepackage{mathrsfs}
\usepackage{comment}
\usepackage{tikz}
\usepackage{svg}
\usepackage{tikz-cd}
\DeclareFontFamily{U}{rsfs}{\skewchar\font127 }
\DeclareFontShape{U}{rsfs}{m}{n}{%
   <-6> rsfs5
   <6-8> rsfs7
   <8-> rsfs10
}{}
\usepackage{bm}
\usepackage{hyperref}
\hypersetup{
	colorlinks=true,
	linkcolor=cyan!50!blue,
	filecolor=blue,      
	urlcolor=red,
	citecolor=blue,
}
\usepackage{stackengine}
\stackMath
\newcommand\subline[2]{\stackon[-1.5pt]{#1}{\rule[2pt]{\widthof{$#1$}}{.4pt}_{#2}}}
\usepackage{cleveref}
\usepackage{makecell}
\usepackage{framed}
\usepackage{graphicx}
\usepackage{xcolor}
\usepackage{listings}
\usepackage{multirow}
\usepackage{multicol}
\usepackage{indentfirst}
\usepackage{booktabs}
\usepackage{tikz}
\usepackage{tikz-cd}
\usepackage{float}
\usepackage{algorithm}
\usepackage{algorithmicx}
\usepackage{algpseudocode}
\usepackage{subfigure}
\usepackage{lineno}
\usepackage{comment}
\usepackage{enumitem}
\usepackage{setspace}
\usetikzlibrary{patterns, arrows.meta, calc}

\newcommand*{\be}[1]{\begin{equation}\label{#1}}
\newcommand*{\ee}{\end{equation}}

\DeclareMathOperator{\Span}{span}

\newtheorem{theorem}{Theorem}[section]

\newtheorem{lemma}{Lemma}[section]
\newtheorem{proposition}{Proposition}[section]

\theoremstyle{remark}
\newtheorem{remark}{Remark}[section]

\definecolor{pink}{RGB}{255,45,115}

\DeclareMathOperator{\grad}{grad}
\DeclareMathOperator{\hess}{hess}
\DeclareMathOperator{\curl}{curl}
\DeclareMathOperator{\inc}{inc}

\DeclareMathOperator{\dev}{dev}
\DeclareMathOperator{\sym}{sym}
\DeclareMathOperator{\diverenge}{div}

\DeclareMathOperator{\tr}{tr}
\newcommand\vskw{\operatorname{vskw}}
\newcommand\mskw{\operatorname{mskw}}

\newcommand\skw{\operatorname{skw}}
\newcommand\sskw{\operatorname{sskw}}

\newcommand{\bS}{\mathbb S}
\newcommand{\bT}{\mathbb T}

\newcommand\K{\mathbb{K}}
\newcommand\M{\mathbb{M}}

\renewcommand\S{{\mathbb S}}

\newcommand\E{{\mathcal{E}}}

\renewcommand{\div}{\diverenge}

\usepackage{geometry}
\newcommand{\R}{\mathbb{R}}

\newcommand{\lt}[1]{{[\color{cyan}Linting:~#1}]}
\newcommand{\kh}[1]{{[\color{blue}KH:~#1}]}

\renewcommand\ker{\mathcal{N}}
 \newcommand{\bs}{{\scriptscriptstyle \bullet}}
 \newcommand\ran{\mathcal{R}}
 \newcommand\alt{\mathrm{Alt}}
 \newcommand\deff{\operatorname{def}}

 \numberwithin{equation}{section}

\usepackage{todonotes}

\DeclareMathOperator{\rot}{rot}
\title{Distributional Hessian and divdiv complexes on triangulation and cohomology}

\begin{document}

\author{Kaibo Hu}
\address{School of Mathematics, the University of Edinburgh, James Clerk Maxwell Building, Peter Guthrie Tait Rd, Edinburgh EH9 3FD, UK.}
\email{kaibo.hu@ed.ac.uk}
\author{Ting Lin}
\address{School of Mathematics, Peking University, Beijing 100871, P.R.China}
\email{lintingsms@pku.edu.cn}
\author{Qian Zhang}
\address{Department of Mathematical Sciences, Michigan Technological University, Houghton, MI 49931, USA}
\email{qzhang15@mtu.edu}

\maketitle 

\begin{abstract}
In this paper, we construct discrete versions of some Bernstein-Gelfand-Gelfand (BGG) complexes, i.e., the Hessian and the divdiv complexes, on triangulations in 2D and 3D. The sequences consist of finite elements with local polynomial shape functions and various types of Dirac measure on subsimplices. The construction generalizes Whitney forms (canonical conforming finite elements) for the de~Rham complex and Regge calculus/finite elements for the elasticity (Riemannian deformation) complex from discrete topological and Discrete Exterior Calculus perspectives. We show that the cohomology of the resulting complexes is isomorphic to the continuous versions, and thus isomorphic to the de~Rham cohomology with coefficients. 
\end{abstract}



\section{Introduction}

Preserving cohomological structures is crucial for reliable and efficient numerical solutions of a large class of numerical PDEs \cite{arnold2018finite,arnold2006finite,arnold2010finite}. For problems involving vector-valued functions and differential forms, important structures are encoded in the de~Rham complex. There are canonical discretizations of the de~Rham complex as discrete differential forms in finite element exterior calculus \cite{arnold2018finite,arnold2006finite,arnold2010finite,hiptmair1999canonical}. In the lowest order case, this coincides with the Whitney forms \cite{bossavit1988whitney,hiptmair2001higher}, and the degrees of freedom of $k$-forms are distributed on $k$-simplices, reflecting a discrete topological structure. See Figure \ref{fig:deRham-whitney} for an illustration of the 3D case. 
\begin{figure}[ht!]
    \centering
    \includegraphics[width=0.75\linewidth]{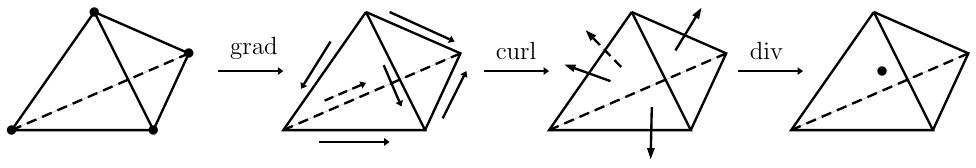}
    \caption{The lowest-order finite element de Rham complex, consisting of Whitney forms (the lowest-order Lagrange, N\'ed\'elec \cite{nedelec1980mixed}, and Raviart-Thomas \cite{raviart2006mixed} elements).}
    \label{fig:deRham-whitney}
\end{figure}
Moreover, discretizations of differential forms are of fundamental interest. Finite element de~Rham complexes lead to Discrete Exterior Calculus (DEC) schemes in some cases when the mass matrix is approximated \cite{desbrun2005discrete,hirani2003discrete} and the idea of discrete forms and cochains can be extended to graphs \cite{lim2020hodge}. These finite elements are conforming in the sense that the $k$-th space is a subspace of $H(d^{k}):=\{u\in L^2\Lambda^{k}: d^{k}u\in L^2\Lambda^{k+1}\}$, where $d^{k}$ is the $k$-th exterior derivative.

Motivated by Equilibrated Residual Error Estimators, Braess and Sch\"oberl \cite{braess2008equilibrated} extended the concept of finite elements by permitting distributions as shape functions. Specifically, they proposed a complex that, in 3D, comprises piecewise constants, Dirac delta in the face normal direction, Dirac delta in the edge tangential direction, and vertex deltas, respectively.  
See Figure \ref{fig:deRham-distributional}.
\begin{figure}
    \centering
    \includegraphics[width=0.7\linewidth]{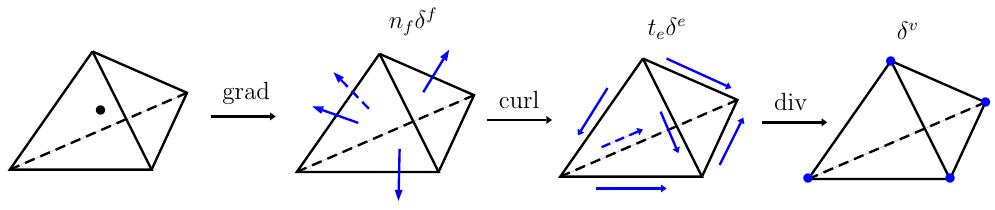}
    \caption{Distributional finite element de~Rham complex in 3D.}
    \label{fig:deRham-distributional}
\end{figure}
The distributional complex (Figure \ref{fig:deRham-distributional}) can be viewed as the dual of the lowest order Whitney forms (Figure \ref{fig:deRham-whitney}).
The construction has been systematically extended to a double complex in \cite{licht2017complexes} and the cohomology is also provided.

There is significant interest in identifying and discretizing differential structures in other problems. A systematic approach is inspired by the Bernstein-Gelfand-Gelfand (BGG) construction \cite{arnold2021complexes,vcap2001bernstein,vcap2022bgg}, where one derives new complexes from existing ones (mostly de~Rham complexes), and the cohomology of the resulting complexes is isomorphic to the input. 

For the three basic examples of the BGG complexes in 3D, i.e., the Hessian, elasticity, and divdiv complexes, conforming finite elements on simplicial meshes in 2D and 3D have been discussed in, e.g., \cite{chen2022finite2D,chen2022finitedivdiv,chen2022finiteelasticity,chen2022complexes,christiansen2023finite,gong2023discrete,hu2021conforming,hu2015family,hu2023nonlinear}. There have also been results in nD \cite{bonizzoni2023discrete,chen2022finite}. These results focus on conforming finite elements and can be viewed as an extension of the study of multivariate (simplicial) splines \cite{lai2007spline} from a homological perspective. Due to the conformity requirement, the results unavoidably involve either high-order polynomials and supersmoothness or special meshes. Discretizations for BGG complexes involving Dirac delta also exist. For the elasticity complex, Christiansen \cite{christiansen2011linearization} interpreted Regge calculus as a finite element fitting in a discrete complex (see Figure \ref{fig:regge}). The metric is in the Regge space, which is a piecewise constant symmetric field with continuous tangential-tangential components. The linearized curvature then consists of Dirac deltas along hinges (edges). The Regge finite element has been extended in \cite{li2018regge} to an arbitrary polynomial degree, and was investigated for discretizing various kinds of curvature tensors \cite{gawlik2023finite,berchenko2022finite,gopalakrishnan2022analysis,gawlik2023finite,gopalakrishnan2023analysis} and solving problems from continuum mechanics and general relativity \cite{li2018regge,neunteufel2021avoiding,neunteufel2023hellan}. For elasticity problems, the TDNNS discretization involving distributional spaces leads to a convergent scheme for the Hellinger-Reissner principle \cite{pechstein2011tangential}. For fluid problems, analogous schemes can be found in \cite{gopalakrishnan2020mass}. For plate problems, spaces in the Hellan-Herrmann-Johnson (HHJ) method can be viewed as a rotation of the 2D Regge element \cite{neunteufel2023hellan,li2018regge}. 
\begin{figure}
    \centering
    \includegraphics[width=0.75\linewidth]{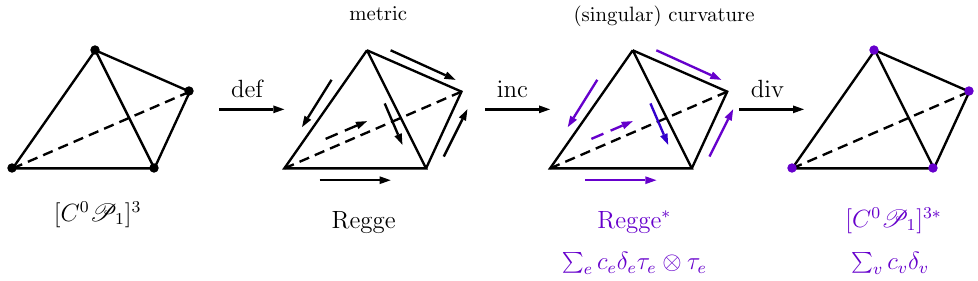}
    \caption{Regge-elasticity complex \cite{christiansen2011linearization}.}
    \label{fig:regge}
\end{figure}

As the Whitney forms both lead to successful finite element schemes and have a neat discrete topological interpretation, it will be natural to ask the question of {\it what are the analogy of the Whitney forms for the BGG complexes}. Answers to this question may extend the discrete topological structures encoded in the de~Rham complexes to other differential structures encoded in BGG, such as discrete Riemannian geometry. The Regge element and complex seem to be a natural candidate for the elasticity complex due to the canonical degrees of freedom, discrete geometric interpretation (in terms of Regge metric) and the duality (the elasticity complex is formally self-adjoint on the continuous level, which is preserved in the Regge complex). 

In this paper, we generalize the Regge complex to the other two BGG complexes in 3D, i.e., the Hessian complex and the divdiv complex, and show that their cohomologies are isomorphic to the continuous versions. This construction is inspired by a Discrete Exterior Calculus observation for the Regge complex. We also obtain the corresponding results in 2D. This thus establishes the cohomology for the Regge complex in 2D.

The rest of the paper will be organized as follows. In \Cref{sec:preliminaries}, we introduce notations of topology, spaces and complexes at the continuous level. In \Cref{sec:main-result}, we summarize the motivation and the main results. The technical details are provided in \Cref{sec:2d} and \Cref{sec:3d}. In Section \ref{sec:conclusion}, we provide concluding remarks and an outlook.

\section{Function spaces and BGG complexes}
\label{sec:preliminaries}
\subsection{Domain and topology}
\label{sec:topology}
Throughout this paper, we assume that $\Omega$ is a bounded Lipschitz domain following the definition of \cite{costabel2010bogovskiui}, i.e., $\Omega$ is a connected bounded open set in $\mathbb{R}^{n}$ which is strong Lipschitz. Let $\Delta$ be a triangulation (simplicial complex) of $\Omega$, and 
we use $\mathsf V, \mathsf E, \mathsf F, \mathsf K$ to denote the set of vertices, edges, faces, and 3D cells, respectively. Note that in 2D, $\mathsf F$ denotes the top (two) dimensional cells.
Their cardinality is denoted as $\sharp_V, \sharp_E, \sharp_F, \sharp_K$, respectively. We use $\mathsf V_{0}$ and $\mathsf V_{\partial}$ to denote the set of internal vertices and boundary vertices, respectively. Their cardinality is denoted as $\sharp_{V_0}$ and $\sharp_{V_\partial}$, respectively. Similar notations are used for edges and faces. 

We refer to standard texts (e.g., \cite{arnold2018finite,hatcher2002algebraic}) for definitions of homology and cohomology. But we provide a brief reminder on simplicial homology $\mathcal H_{\bs}(\Delta)$ and its relative version $\mathcal H_{\bs}(\Delta; \partial \Delta)$, together with their relations to the de Rham cohomologies, which will be used in the rest of the paper.

We first introduce the orientation function. Given a $k$-simplex $\sigma = [ x_1, x_2,\cdots, x_k]$ and a $(k-1)-$simplex $\tau$, define 
\begin{equation}
\mathcal O^{\mathcal H_{\bs}(\Delta)}(\tau, \sigma) = \left\{\begin{aligned} (-1)^j & \text{ if } \tau = [ x_1,  x_2, \cdots, \widehat{ x_j}, \cdots,  x_k] \text{ for some index }j, \\ 0 & \text{ otherwise. } 
\end{aligned}\right.
\end{equation}
 {Here, the superscript means that this orientation is used to compute the simplicial homology $\mathcal H_{\bs}(\Delta)$ (distinguished from the relative version below). }
Now, we fix a vector space $V$, and consider the free abelian group $C_k(\Delta, V)$ with basis $\Delta_k$, the $k$-simplices of the complex $\Delta$, and coefficients in $V$. Then the simplicial boundary operator $\partial_k: C_k(\Delta, V) \to C_{k-1}(\Delta, V)$ can be written as 
\begin{equation}
\partial^{\mathcal H_{\bs}(\Delta)}_k \sigma = \sum_{\tau \in \Delta_{k-1}} \mathcal O^{\mathcal H_{\bs}(\Delta)}(\tau, \sigma) \tau,
\end{equation}
for any $k$-simplex $\sigma$. 
A straightforward calculation yields that $\partial_{k-1} \circ \partial_{k} = 0$, and the standard simplicial homology is then defined by $$\mathcal H_k(\Delta,V) := \ker(\partial_{k}^{\mathcal H_{\bs}(\Delta)})/ \ran(\partial_{k-1}^{\mathcal H_{\bs}(\Delta)}).$$

In this paper, we will also use a special kind of relative homology, $\mathcal H_{\bs}(\Delta; \partial \Delta).$ Here, the chain group $C_k(\Delta, V; \partial \Delta)$ is generated by $k$-simplices in $\Delta \setminus \partial \Delta$. Given a $k$-simplex $\sigma = [ x_1, x_2,\cdots,  x_k]$ and a $(k-1)-$simplex $\tau$, define 
\begin{equation}
\mathcal O^{\mathcal H_{\bs}(\Delta; \partial \Delta)}(\tau, \sigma) = \left\{\begin{aligned} \mathcal O^{\mathcal H_{\bs}(\Delta)}(\tau, \sigma) & \text{ if } \tau \not\in \partial \Delta, \\ 0 & \text{ otherwise. } 
\end{aligned}\right.
\end{equation}
Similarly, we define 
\begin{equation}
    \partial^{\mathcal H_{\bs}(\Delta;\partial \Delta)}_k \sigma = \sum_{\tau \in \Delta_{k-1}} \mathcal O^{\mathcal H_{\bs}(\Delta; \partial \Delta)}(\tau, \sigma)\tau,
    \end{equation}
and $$\mathcal H_k(\Delta, V; \partial \Delta) = \ker(\partial_{k}^{\mathcal H_{\bs}(\Delta; \partial \Delta)})/ \ran(\partial_{k-1}^{\mathcal H_{\bs}(\Delta; \partial \Delta)}).$$

We simply use the notations $\mathcal H_{\bs}(\Delta)$ and $\mathcal H_{\bs}(\Delta; \partial \Delta)$ when $V = \mathbb Z$. Two theorems from algebraic topology will be relevant for the discussions below. The first is the universal coefficient theorem, computing the homology with coefficients in $V$. 
\begin{theorem}[Universal coefficient theorem, simplified version]\label{UCT}
For any vector space $V$, it holds that 
\begin{equation}
\mathcal H_{\bs}(\Delta,V) \cong \mathcal H_{\bs}(\Delta) \otimes V,
\end{equation}
and 
\begin{equation}
    \mathcal H_{\bs}(\Delta,V; \partial \Delta) \cong \mathcal H_{\bs}(\Delta; \partial \Delta) \otimes V.
\end{equation}
\end{theorem}

The next celebrated result is the Lefschetz duality, which relates the relative homology to the de~Rham cohomology. 
\begin{theorem}[Lefschetz duality theorem]
Let $\Delta$ be a $n$-complex with boundary $\partial \Delta$. Then it holds that 
\begin{equation}
\mathcal H_k(\Delta; \partial \Delta) \cong \mathcal H^{n-k}(\Delta),
\end{equation}
and 
\begin{equation}
\mathcal H_k(\Delta) \cong \mathcal H^{n-k}(\Delta; \partial \Delta),
\end{equation}
Here $\mathcal H^{n-k}$ is the simplicial cohomology.
\end{theorem}

Now we consider the relation to de Rham cohomology. We use $\mathcal H^{\bs}_{dR}(\Omega)$ to denote the standard de Rham cohomology, and $\mathcal H^{\bs}_{dR,c}(\Omega)$ to denote the compactly supported version. Using the de~Rham theorem, we get the following result. 
\begin{theorem}\label{thm:deRham}
The following results hold.

\begin{enumerate}
    \item $\mathcal H_k(\Delta; \partial \Delta) \cong \mathcal H^{n-k}_{dR}(\Omega).$
    \item $\mathcal H_k(\Delta) \cong \mathcal H^{n-k}_{dR,c}(\Omega).$
\end{enumerate}
\end{theorem}



Unless otherwise specified, we use $\mathcal O$ to denote the orientation with respect to the relative homology $\mathcal O^{\mathcal H_{\bs}(\Delta; \partial \Delta)}$, use $\mathcal O_0$ to denote $\mathcal O^{\mathcal H_{\bs}(\Delta)}$, in accordance with the compact supported de Rham cohomology.

In the rest of this paper, we will relate functions and distributions to an element in the chain group $C_k(\Delta)$. To this end, we denote by $\|\sigma\|$ the free element in $C_k(\Delta)$ associated with $k$-th simplex $\sigma$. This does not lead to confusion since we will not use $\|\sigma\|$ as a norm.

\subsection{Function spaces}

In this subsection, we define various spaces of vector-valued and matrix-valued functionals.

We introduce some scalar spaces following \cite{costabel2010bogovskiui}. Recall that $\Omega$ is a bounded Lipschitz domain. Let $s$ be any real number. Then $H^{s}(\Omega)$ denotes the quotient of $H^{s}(\mathbb{R}^{n})$ with distributions vanishing in $\Omega$, which is isomorphic to the standard definition as a space of distributions in $\Omega$. Throughout this paper, we will simply consider $H^s(\Omega)$ as a space of distributions over $\Omega$. 

Let $H_{\overline{\Omega}}^{s}(\mathbb{R}^{n})$ be the subspace of $H^{s}(\mathbb{R}^{n})$ consisting of all the distributions with support in $\overline{\Omega}$. For any $s$, $H_{\overline{\Omega}}^{s}(\mathbb{R}^{n})$ can be identified with the closure of $C_{c}^{\infty}$ in $H^{s}(\mathbb R^{n})$, and also the dual space of $H^{-s}(\Omega)$, i.e., 
{\renewcommand\stackalignment{l}
\begin{equation}\label{H0-space}
H_{\overline{\Omega}}^{s}(\mathbb{R}^{n})=\subline{C_{c}^{\infty}}{H^{s}(\mathbb R^{n})}=[H^{-s}(\Omega)]^{\ast}, \quad \forall s\in \mathbb R.
\end{equation}
}
We also have
$$
H_{\overline{\Omega}}^{s}(\mathbb{R}^{n})=H^{s}_{0}(\Omega), \quad \forall \mbox{ $s>0$ and $s-\frac{1}{2}$ is not an integer,}
$$
where $H^{s}_{0}(\Omega)$ denotes the closure of $C^{\infty}_{c}(\Omega)$ in $H^{s}(\Omega)$. 

Let $H^{s}\Lambda^{k}$ (or $H^{s}_{\overline\Omega}\Lambda^{k}$) be the space of differential $k$-forms with coefficients in $H^{s}(\Omega)$ (or correspondingly, $H^{s}_{\overline\Omega}(\mathbb{R}^{n})$). Following \cite{costabel2010bogovskiui}, we have two complexes:
\begin{equation}\label{Hs-complex}
\begin{tikzcd}
  0\arrow{r}{}&H^s\Lambda^0\arrow{r}{d^0}& H^{s-1}\Lambda^1\arrow{r}{d^1} & \cdots \arrow{r}{d^{n-1}} & H^{s-n}\Lambda^n\arrow{r}{}&0,
\end{tikzcd}
\end{equation}
\begin{equation}\label{Hs0-complex}
\begin{tikzcd} 
  0\arrow{r}{}&H^s_{\overline \Omega}\Lambda^0\arrow{r}{d^0}& H^{s-1}_{\overline \Omega}\Lambda^1\arrow{r}{d^1} & \cdots \arrow{r}{d^{n-1}} & H^{s-n}_{\overline \Omega}\Lambda^n\arrow{r}{}&0.
\end{tikzcd}
\end{equation}
The complexes \eqref{Hs-complex} and \eqref{Hs0-complex} both have uniform cohomology in the sense that the cohomology is isomorphic to the smooth versions. 

\begin{remark}
\label{rmk:subspace-diff}
Note that the differential operators in the complexes actually have different meanings. In \eqref{Hs-complex}, we consider distributions in $\Omega$, where the test function is chosen as $C_c^{\infty}(\Omega)$. While in \eqref{Hs0-complex}, we consider distributions in $\mathbb R^n$, where the test function is in $C_c^{\infty}(\mathbb R^n).$ 
\end{remark}

Most of the examples in this paper will be given in vector/matrix forms. We introduce some notation following  \cite{arnold2021complexes,vcap2022bgg}.
\begin{table}[h!]
\begin{center}
\begin{tabular}{c|c}
$\mathbb V$ & $\mathbb R^n$\\
$\mathbb M$ &the space of all $n\times n$-matrices\\
$\mathbb S$ & symmetric matrices\\
$\mathbb K$ & skew symmetric matrices\\
$\mathbb T$ & trace-free matrices\\
$\skw: \M\to \K$ & skew symmetric part of a matrix\\
$\sym: \M\to \S$ & symmetric part of a matrix\\
$\tr:\M\to\R$ & matrix trace\\
$\iota: \R\to \M$  & the map $\iota u:= uI$ identifying a scalar with a scalar matrix\\
$\dev:\mathbb{M}\to \mathbb{T}$ & deviator (trace-free part of a matrix) given by $\dev u:=u-1/n \tr (u)I$\\

\end{tabular}
\end{center}
\end{table}

We use $\mathcal D'$ to denote the space of distributions and $\langle \cdot, \cdot\rangle$ for dual pairs. We use $(\cdot, \cdot)$ to denote the $L^{2}$ inner product. This can be naturally extended to vector- or matrix-valued functions. For example, $(\sigma, \tau):=\int \sigma:\tau\, dx$ when $\sigma$ and $\tau$ are both matrix fields. The above matrix operations can be extended to functions and distributions naturally. For example, we define $\sym: \mathcal{D}'(\mathbb{M})\to \mathcal{D}'(\mathbb{S})$ by 
$ \langle \sym  \xi,  \sigma \rangle = \langle \xi, \sym  \sigma \rangle,\quad \forall  \sigma \in \mathcal{D}(\mathbb{M}),$
and define $\dev : \mathcal D'(\mathbb M) \to \mathcal D'(\mathbb T)$ by $\langle \dev \xi, \sigma \rangle = \langle \xi, \dev \sigma \rangle, \quad \forall \sigma \in \mathcal D(\mathbb M).$

  Let  $\mathbb{X}_{i}$ be one of the above vector or matrix spaces ($\mathbb{V}, \mathbb{M}, \mathbb{S}$, etc.) and $D^{i}$ be a linear differential operator with constant coefficients of order $\alpha_{i}$. Suppose that we have the following short sequence 
\begin{equation}\label{short-sequence}
\begin{tikzcd}
\cdots\arrow{r}{}& H^{s_{0}, s_1}(D^{1}, \mathbb{X}_{1})\arrow{r}{D^{1}} & H^{s_{1}, s_{2}}(D^{2}, \mathbb{X}_{2})\arrow{r}{D^{2}} & H^{s_{2}, s_{3}}(D^{3}, \mathbb{X}_{3})\arrow{r}{} & \cdots,
\end{tikzcd}
\end{equation}
where 
$$
H^{s_{1}, s_{2}}(D; \mathbb{X}_{i}):=\{u\in H^{s_{1}}\otimes \mathbb{X}_{i}: Du\in H^{s_{2}}\otimes \mathbb{X}_{i+1}\}.
$$
Similarly, we define 
$$
H_{0}^{s_{1}, s_{2}}(D; \mathbb{X}_{i}):=\{u\in H_{\overline\Omega}^{s_{1}}\otimes \mathbb{X}_{i}: Du\in H_{\overline\Omega}^{s_{2}}\otimes \mathbb{X}_{i+1}\},
$$
and they also form a complex. When $s_{1}=s_{2}$, we denote $H^{s}(D; \mathbb{X}_{i}):=H^{s, s}(D; \mathbb{X}_{i})$, and when $s=0$, we further denote $H(D; \mathbb{X}_{i}):=H^{0}(D; \mathbb{X}_{i})$ and $H_{0}(D; \mathbb{X}_{i}):=H^{0}_{0}(D; \mathbb{X}_{i})$.

To state regularity and duality results, we assume that the sequence
\begin{equation}
\begin{tikzcd}
\cdots\arrow{r}{}&H^{q}\otimes  \mathbb{X}_{1}\arrow{r}{D^{1}} & H^{q-\alpha_{1}} \otimes\mathbb{X}_{2}\arrow{r}{D^{2}} & H^{q-\alpha_{1}-\alpha_{2}} \otimes\mathbb{X}_{3}\arrow{r}{D^{3}} & \cdots,
\end{tikzcd}
\end{equation}
is a Hilbert scale in the sense that for any $q\in \mathbb{R}$, we have 
\begin{equation}\label{uniform-cohomology-1}
\ker(D^{2},  H^{q-\alpha_{1}} \otimes\mathbb{X}_{2})=(D^{1}H^{q}\otimes  \mathbb{X}_{1})\oplus\mathscr{H}^{2},
\end{equation}
and
\begin{equation}\label{uniform-cohomology-2}
\ker(D^{3},  H^{q-\alpha_{1}-\alpha_{2}} \otimes\mathbb{X}_{3})=(D^{2}H^{q-\alpha_{1}}\otimes  \mathbb{X}_{2})\oplus\mathscr{H}^{3},
\end{equation}
where $\mathscr{H}^{2}$ and $\mathscr{H}^{3}$ are finite dimensional spaces not depending on $q$ (thus containing smooth functions) \cite{arnold2021complexes}.  

Let $(D^{1})^{\ast}$ be the formal adjoint of $D^{1}$. 
 Related to \eqref{short-sequence}, we have the following, generalizing the regular decomposition argument \cite[Theorem 3]{arnold2021complexes} and the duality result in \cite[Lemma 1]{pechstein2011tangential}.

\begin{theorem}\label{thm:decomp-duality}
 We have
\begin{equation}\label{regular-decomposition}
H^{s_{1}, s_2}(D^{2}; \mathbb{X}_{2})=H^{s_{2}+\alpha_{2}}\otimes \mathbb{X}_{2} +D^{1}H^{s_{1}+\alpha_{1}}(D^{1}, \mathbb{X}_{1}),
\end{equation}
and the duality
\begin{equation}\label{duality-sobolev}
H_{0}^{s_{1}, s_{2}}(D^2, \mathbb{X}_{2})^{\ast}=H^{-s_{2}-\alpha_{2}, -s_{1}-\alpha_{1}}((D^{1})^{\ast}, \mathbb{X}_{2}).
\end{equation}
Similarly, for another type of boundary conditions, we have
$$
H_{0}^{s_{1}, s_2}(D^{2}; \mathbb{X}_{2})=H_{0}^{s_{2}+\alpha_{2}}\otimes \mathbb{X}_{2} +D^{1}H_{0}^{s_{1}+\alpha_{1}}(D^{1}, \mathbb{X}_{1}),
$$
and the duality
$$
H^{s_{1}, s_{2}}(D^2, \mathbb{X}_{2})^{\ast}=H_{0}^{-s_{2}-\alpha_{2}, -s_{1}-\alpha_{1}}((D^{1})^{\ast}, \mathbb{X}_{2}).
$$
\end{theorem}

\begin{proof}
For any $u\in H^{s_{1}, s_{2}}(D^{2}; \mathbb{X}_{2})$, $D^{2}u\in H^{s_{2}}\otimes \mathbb{X}_{3}$. Due to \eqref{uniform-cohomology-2}, there exists $\varphi\in H^{s_{2}+\alpha_{2}}\otimes \mathbb{X}_{2}$  such that $D^{2}\varphi=D^{2}u$, i.e., $D^{2}(u-\varphi)=0$. Using \eqref{uniform-cohomology-1}, we have 
$$
u-\varphi=D^{1}\beta+h,
$$
where $\beta\in H^{s_{1}+\alpha_{1}}\otimes \mathbb{X}_{1}$ and $h$ is smooth. This proves \eqref{regular-decomposition}.

For $w\in H_{0}^{s_{1}, s_{2}}(D^2, \mathbb{X}_{2})^{\ast}$, 
\begin{align*}
\|w\|_{H_{0}^{s_{1}, s_{2}}(D^2, \mathbb{X}_{2})^{\ast}}&=\sup_{v\in H_{0}^{s_{1}, s_{2}}(D^2, \mathbb{X}_{2}) } \frac{\langle v, w\rangle}{\|v\|_{H_{0}^{s_{1}, s_{2}}(D^2, \mathbb{X}_{2})}}\\&
\cong\sup_{y\in H^{s_{1}+\alpha_{1}}_{0}\otimes \mathbb{X}_{1},  z\in H^{s_{2}+\alpha_{2}}_{0}\otimes \mathbb{X}_{2}}\frac{ \langle D^{1}y+z, w\rangle}{\|y\|_{s_{1}+\alpha_{1}}+\|z\|_{s_{2}+\alpha_{2}}}\\
&\cong\sup_{y\in H^{s_{1}+\alpha_{1}}_{0}\otimes \mathbb{X}_{1},  z\in H^{s_{2}+\alpha_{2}}_{0}\otimes \mathbb{X}_{2}}\frac{ \langle D^{1}y, w\rangle}{\|y\|_{s_{1}+\alpha_{1}}}+\frac{\langle z, w\rangle}{\|z\|_{s_{2}+\alpha_{2}}}\\
&=\sup_{y\in H^{s_{1}+\alpha_{1}}_{0}\otimes \mathbb{X}_{1},  z\in H^{s_{2}+\alpha_{2}}_{0}\otimes \mathbb{X}_{2}}\frac{ \langle y, (D^{1})^{\ast}w\rangle}{\|y\|_{s_{1}+\alpha_{1}}}+\frac{\langle z, w\rangle}{\|z\|_{s_{2}+\alpha_{2}}}\\
&=\|(D^{1})^{\ast}w\|_{-s_{1}-\alpha_{1}}+\|w\|_{-s_{2}-\alpha_{2}}.
\end{align*}
\end{proof}

As special cases of Theorem \ref{thm:decomp-duality}, we recover results from the literature (c.f. \cite{pechstein2011tangential})
$$
H_{0}(\curl)^{\ast}=H^{-1}(\div), \quad H_{0}(\div)^{\ast}=H^{-1}(\curl), 
$$
and obtain new results. For example, associated with the Hessian complex, we have
$$
H_{0}(\curl; \mathbb{S})^{\ast}=H^{-1, -2}(\div\div; \mathbb{S}):=\{w\in H^{-1}\otimes \mathbb{S} : \div\div w\in H^{-2}\},
$$
and associated with the elasticity complex, 
$$
H_{0}(\inc; \mathbb{S})^{\ast}=H^{-2, -1}(\div; \mathbb{S}):=\{w\in H^{-2}\otimes \mathbb{S} : \div w\in H^{-1}\}.
$$

\subsection{BGG complexes}

  The BGG construction is a machinery to derive other complexes from the de~Rham complexes. We follow and review the constructions in \cite{arnold2021complexes,vcap2022bgg}. Even though we only focus on complexes in 2D and 3D in this case,  the more general differential form point of view shall provide discrete topological inspirations and interpretations for our construction in this paper.


For $i\ge 0$, let $\alt^i \R^n$ be the space of algebraic $i$-forms, that is, of alternating $i$-linear maps on $\mathbb{R}^{n}$.  We also set $\alt^{i, j}\R^n=\alt^{i}\R^n\otimes \alt^{j}\R^n$, the space of $\alt^{j}\R^n$-valued $i$-forms or, equivalently, the space of $(i+j)$-linear maps on $\mathbb{R}^{n}$ which are alternating in the first $i$ variables and also in the last $j$ variables. A major example in \cite{arnold2021complexes} is constructed from the following diagram:
\begin{equation}\label{diagram-nD}
\begin{tikzcd}[column sep=tiny]
0 \arrow{r} & H^{q}\otimes\alt^{0,0}  \arrow{r}{d} &H^{q-1}\otimes\alt^{1,0}  \arrow{r}{d} & \cdots \arrow{r}{d} & H^{q-n}\otimes \alt^{n,0} \arrow{r}{} & 0\\
0 \arrow{r} & H^{q-1}\otimes\alt^{0,1}  \arrow{r}{d} \arrow[ur, "S^{0,1}"] &H^{q-2}\otimes\alt^{1,1}  \arrow{r}{d}  \arrow[ur, "S^{1,1}"] & \cdots \arrow{r}{d}  \arrow[ur, "S^{n-1,1}"] & H^{q-n-1}\otimes \alt^{n,1} \arrow{r}{} & 0\\[-15pt]
 & \vdots & \vdots & {} & \vdots & {} \\[-15pt]
0 \arrow{r} & H^{q-n+1}\otimes\alt^{0,n-1}  \arrow{r}{d} &H^{q-n}\otimes\alt^{1,n-1}  \arrow{r}{d} & \cdots \arrow{r}{d} & H^{q-2n+1}\otimes \alt^{n,n-1} \arrow{r}{} & 0\\
0 \arrow{r} & H^{q-n}\otimes\alt^{0,n}  \arrow{r}{d} \arrow[ur, "S^{0,n}"] &H^{q-n-1}\otimes\alt^{1,n}  \arrow{r}{d}  \arrow[ur, "S^{1,n}"] & \cdots \arrow{r}{d}  \arrow[ur, "S^{n-1,n}"] & H^{q-2n}\otimes \alt^{n,n} \arrow{r}{} & 0.
\end{tikzcd}
\end{equation}
Here  $s^{i, j}: \alt^{i, j} \R^n\to \alt^{i+1, j-1}\R^n$ are algebraic maps, and $S^{i,j}=I\otimes s^{i,j}:H^q\otimes\alt^{i, j}\to H^q\otimes\alt^{i+1, j-1}$ for any Sobolev order $q$. The $k$-th row of \eqref{diagram-nD} is obtained by tensoring a de~Rham complex with alternating $k$-forms.

The BGG diagram~\eqref{diagram-nD} has the following vector/matrix proxies for $n=3$:
 \begin{equation}\label{diagram-4rows}
\begin{tikzcd}
0 \arrow{r} &H^{q}\otimes \mathbb{R}  \arrow{r}{\grad} &H^{q-1}\otimes \mathbb{V} \arrow{r}{\curl} &H^{q-2}\otimes \mathbb{V} \arrow{r}{\div} & H^{q-3}\otimes \mathbb{R} \arrow{r}{} & 0\\
0 \arrow{r}&H^{q-1}\otimes \mathbb{V}\arrow{r}{\grad} \arrow[ur, "I"]&H^{q-2}\otimes \mathbb{M}  \arrow{r}{\curl} \arrow[ur, "2\vskw"]&H^{q-3}\otimes \mathbb{M} \arrow{r}{\div}\arrow[ur, "\tr"] & H^{q-4}\otimes \mathbb{V} \arrow{r}{} & 0\\
0 \arrow{r} &H^{q-2}\otimes \mathbb{V}\arrow{r}{\grad} \arrow[ur, "-\mskw"]&H^{q-3}\otimes \mathbb{M}  \arrow{r}{\curl} \arrow[ur, "\mathcal{S}"]&H^{q-4}\otimes \mathbb{M} \arrow{r}{\div}\arrow[ur, "2\vskw"] & H^{q-5}\otimes \mathbb{V} \arrow{r}{} & 0\\
0 \arrow{r} &H^{q-3}\otimes \mathbb{R}\arrow{r}{\grad} \arrow[ur, "\iota"]&H^{q-4}\otimes \mathbb{V}  \arrow{r}{\curl} \arrow[ur, "-\mskw"]&H^{q-5}\otimes \mathbb{V} \arrow{r}{\div}\arrow[ur, "I"] & H^{q-6}\otimes \mathbb{R} \arrow{r}{} & 0,
 \end{tikzcd}
\end{equation}
where, for any matrix $u$,  $\mathcal{S}u=u^{T}-\tr(u)I$.

In 3D, there are three basic examples of BGG complexes \cite{arnold2021complexes,vcap2022bgg}:
the {Hessian complex} 
\begin{equation}\label{grad-grad0}
\begin{tikzcd}
0\arrow{r}  & H^{q}\otimes \mathbb{R} \arrow{r}{\hess} & H^{q-2}\otimes \mathbb{S} \arrow{r}{\curl} & H^{q-3}\otimes \mathbb{T} \arrow{r}{\div} & H^{q-4}\otimes \mathbb{V} \arrow{r} & 0,
\end{tikzcd}
\end{equation}
where $\hess:=\grad\grad$;
the {elasticity complex}
 \begin{equation}\label{sequence:hs}
\begin{tikzcd}
0\arrow{r} & H^{q-1}\otimes \mathbb{V} \arrow{r}{{\deff}} & H^{q-2}\otimes \mathbb{S} \arrow{r}{\inc} & H^{q-4}\otimes \mathbb{S} \arrow{r}{\div} & H^{q-5}\otimes \mathbb{V} \arrow{r} & 0,
\end{tikzcd}
\end{equation}
where $\deff:=\sym\grad$ is the linearized deformation (symmetric part of gradient) and $\inc = \curl \mathcal{S}^{-1}\curl$ leads to the linearized Einstein tensor;
and the {$\div\div$ complex}
 \begin{equation}\label{div-div0}
\begin{tikzcd}
0 \arrow{r}& H^{q-2}\otimes \mathbb{V}  \arrow{r}{\dev\grad} & H^{q-3}\otimes \mathbb{T}  \arrow{r}{\sym\curl} & H^{q-4}\otimes \mathbb{S}  \arrow{r}{\div\div} & H^{q-6}\otimes \mathbb{V} \arrow{r} &0.
\end{tikzcd}
\end{equation}
The three basic examples \eqref{grad-grad0}-\eqref{div-div0} are  derived from \eqref{diagram-4rows}. The first two rows imply the Hessian complex \eqref{grad-grad0}; the middle two rows derive the elasticity complex \eqref{sequence:hs}; the last two rows yield the $\div\div$ complex \eqref{div-div0}. The complexes \eqref{grad-grad0}-\eqref{div-div0} are resolutions of finite dimensional spaces 
$$
\mathcal{P}_1=\ker (\hess), \quad \mathcal{RM}:=\ker (\deff), \quad\mathcal{RT}:=\{\bm a+b\bm x: \bm a\in \mathbb{R}^n, b\in \mathbb{R}\}=\ker (\dev\grad), 
$$
respectively. Here $ \mathcal{RM}$ and $ \mathcal{RT}$ refer to infinitesimal rigid body motions and the Raviart-Thomas \cite{raviart2006mixed} shape functions.
The cohomology of \eqref{grad-grad0}-\eqref{div-div0} is isomorphic to the de~Rham cohomology with $\mathcal{P}_1$, $\mathcal{RM}$, and $\mathcal{RT}$ coefficients, respectively.

The general idea of deriving a BGG complex from a diagram is to eliminate spaces connected by the algebraic operators as much as possible. We use the elasticity complex \eqref{sequence:hs} to explain this construction. The first connecting map $-\mskw$ is surjective. Therefore we eliminate $H^{q-2}\otimes \mathbb{V}$ and the image of $-\mskw$ (as a subspace of $H^{q-2}\otimes \mathbb{M}$). The remaining part in $H^{q-2}\otimes \mathbb{M}$ is thus symmetric matrices. The second connecting map, $\mathcal{S}$, is bijective. Hence we eliminate the two spaces connected by $\mathcal S$ from the diagram and connect the two complexes by a zig-zag $\curl \circ \mathcal S^{-1}\circ \curl$. Finally, the last operator $2\vskw$ is surjective. We eliminate $H^{q-4}\otimes \mathbb{V}$ and the corresponding part in $H^{q-4}\otimes \mathbb{M}$. Again, this leaves symmetric matrices in the derived complex.

In $\mathbb{R}^{2}$, a skew-symmetric matrix can be identified with a scalar. Using the same notation as in 3D, define $\mskw: \mathbb{R}\to \mathbb{K}$ by 
$$
\mskw(u):= \left ( 
\begin{array}{cc}
0 &u\\
-u & 0
\end{array}
\right )\quad \mbox{in } \mathbb{R}^{2}.
$$
Let $\sskw=\mskw^{-1}\circ \skw: \mathbb{M}\to \mathbb{R}$ be the map taking the skew part of a matrix and identifying it with a scalar (see \cite{arnold2021complexes}).

The 2D version of the diagram \eqref{diagram-4rows} is
\begin{equation}\label{diagram-3rows2D}
\begin{tikzcd}
0 \arrow{r}  &H^{q}\otimes \mathbb{R} \arrow{r}{\grad} &H^{q-1}\otimes \mathbb{V} \arrow{r}{\rot} & H^{q-2}\otimes \mathbb{R} \arrow{r}{} & 0\\
0 \arrow{r} & H^{q-1}\otimes \mathbb{V} \arrow{r}{\grad}\arrow[ur, "I"] &H^{q-2}\otimes \mathbb{M} \arrow{r}{\rot}\arrow[ur, "-2\sskw"] & H^{q-3}\otimes \mathbb{V} \arrow{r}{} & 0\\
0 \arrow{r}  &H^{q-2}\otimes \mathbb{R} \arrow{r}{\grad}\arrow[ur, "\mskw"] &H^{q-3}\otimes \mathbb{V} \arrow{r}{\rot}\arrow[ur, "\mathrm{id}"] & H^{q-4}\otimes \mathbb{R} \arrow{r}{} & 0.
 \end{tikzcd} 
\end{equation}
The derived complexes are a rotated stress complex
\begin{equation} 
\begin{tikzcd}
0\arrow{r}  &H^{q} \arrow{r}{\hess} & H^{q-2}\otimes \mathbb{S}  \arrow{r}{\rot} &H^{q-3}\otimes \mathbb{V}   \arrow{r}{} & 0,
\end{tikzcd}
\end{equation}
and the strain complex
\begin{equation} 
\begin{tikzcd}
0\arrow{r}  &H^{q-1} \otimes \mathbb{V} \arrow{r}{\sym\grad} & H^{q-2}\otimes \mathbb{S}  \arrow{r}{\rot\rot} &H^{q-4}   \arrow{r}{} & 0.
\end{tikzcd}
\end{equation}

The BGG framework provides a differential form interpretation for the Hessian, elasticity, and $\div\div$ complexes. In 3D, they can be sketched as follows:
\begin{equation} 
\begin{tikzcd}
0\arrow{r}  &\Lambda^{0, 0} \arrow{r}{d\circ I\circ d} & \Lambda^{1, 1} \arrow{r}{d} & \Lambda^{2, 1} \arrow{r}{d} & \Lambda^{3, 1}\arrow{r} & 0,
\end{tikzcd}
\end{equation}
\begin{equation} \label{elasticity-Lambda}
\begin{tikzcd}
0\arrow{r}  &\Lambda^{0, 1} \arrow{r}{d} & \Lambda^{1, 1} \arrow{r}{d\circ \mathcal S^{-1}\circ d} & \Lambda^{2, 2} \arrow{r}{d} & \Lambda^{3, 2}\arrow{r} & 0,
\end{tikzcd}
\end{equation}
\begin{equation} 
\begin{tikzcd}
0\arrow{r}  &\Lambda^{0, 2} \arrow{r}{d} & \Lambda^{1, 2} \arrow{r}{d} & \Lambda^{2, 2} \arrow{r}{d\circ I\circ d} & \Lambda^{3, 3}\arrow{r} & 0.
\end{tikzcd}
\end{equation}

In this paper, we mainly consider the following Sobolev versions of the BGG complexes:
 \begin{equation} \label{cplx:hessian-sobolev}
\begin{tikzcd}
0\arrow{r}  &H^1 \arrow{r}{\hess} & H^{-1}(\curl, \mathbb{S}) \arrow{r}{\curl} & H^{-1}(\div, \mathbb{T}) \arrow{r}{\div} & H^{-1}\otimes \mathbb{V}\arrow{r} & 0,
\end{tikzcd}
\end{equation}
 \begin{equation} \label{cplx:elasticity-sobolev}
\begin{tikzcd}
0\arrow{r}  &H^1\otimes \mathbb{V} \arrow{r}{\deff} & H^{0, -1}(\inc, \mathbb{S}) \arrow{r}{\inc} & H^{-1}(\div, \mathbb{S}) \arrow{r}{\div} & H^{-1}\otimes \mathbb{V}\arrow{r} & 0,
\end{tikzcd}
\end{equation}
 \begin{equation} \label{cplx:divdiv-sobolev}
\begin{tikzcd}
0\arrow{r}  &H^1 \arrow{r}{\dev\grad} & H(\sym\curl, \mathbb{S}) \arrow{r}{\sym\curl} & H^{0, -1}(\div\div, \mathbb{T}) \arrow{r}{\div\div} & H^{-1}\otimes \mathbb{V}\arrow{r} & 0.
\end{tikzcd}
\end{equation}
Note that \eqref{cplx:hessian-sobolev}-\eqref{cplx:divdiv-sobolev} have a symmetric pattern. First, we observe that $H^{1}=H^{-1}(\hess)$ (by the Lions lemma \cite{lions2012non,ciarlet2013linear}, $\hess u\in H^{-1}$ implies that $u\in H^{1}$), and thus all the spaces in \eqref{cplx:hessian-sobolev} has the form of $H^{-1}(D)$, where $D$ is the corresponding BGG differential operator. Second, due to Theorem \ref{thm:decomp-duality}, the divdiv complex \eqref{cplx:divdiv-sobolev} is dual (adjoint) to the compactly supported version of \eqref{cplx:hessian-sobolev}. Finally, the elasticity complex \eqref{cplx:elasticity-sobolev} has a self-dual pattern: $H^1(\mathbb{V})^{\ast}=H_0^{-1}(\mathbb{V})$, and $H^{0, -1}(\inc, \mathbb{S})^{\ast}=H^{-1}_0(\div, \mathbb{S})$. 

All these complexes also have compactly supported versions. For example, another version of \eqref{cplx:hessian-sobolev} is
 \begin{equation} \label{eq:hessian_lowreg-exact}
\begin{tikzcd}
0\arrow{r}  &H^1_{0} \arrow{r}{\hess} & H^{-1}_{0}(\curl, \mathbb{S}) \arrow{r}{\curl} & H^{-1}_{0}(\div, \mathbb{T}) \arrow{r}{\div} & H^{-1}_{0}\otimes \mathbb{V}\arrow{r} & 0.
\end{tikzcd}
\end{equation}

To clarify notations in 2D, we define $\curl u = [-\frac{\partial u}{\partial y}, \frac{\partial u}{\partial x}]$, and $\rot([w_x, w_y]) = \frac{\partial w_y}{\partial x} - \frac{\partial w_x}{\partial y}.$ By definition, $\rot$ and $-\curl$ are formal adjoints to each other. Moreover, we define $([x,y])^{\perp} = [-y, x]$. 
Therefore, $\curl u = (\grad u)^{\perp}$ and $\rot(\bm w) = \div(\bm w^{\perp})$. For each edge we assume that $\bm n_e = \bm t_e^{\perp}$. It is easy to see that $\curl u \cdot \bm n_e = \grad u \cdot \bm t_e$. For matrices, we define $\curl$ and $\rot$ rowwise.

In 2D, we consider the (rotated) stress complex
 \begin{equation} \label{cplx:stress-sobolev}
\begin{tikzcd}
0\arrow{r}  &H^1 \arrow{r}{\hess} & H^{-1}(\curl, \mathbb{S}) \arrow{r}{\rot} & H^{-1}(\mathbb{V})\arrow{r}{} &  0,
\end{tikzcd}
\end{equation}
and the stain complex
 \begin{equation} \label{cplx:strain-sobolev}
\begin{tikzcd}
0\arrow{r}  &H^1 \arrow{r}{\deff} & H^{0, -1}(\rot\rot, \mathbb{S})\arrow{r}{\rot\rot} & H^{-1}\arrow{r} & 0.
\end{tikzcd}
\end{equation}
Now the two complexes \eqref{cplx:stress-sobolev} and \eqref{cplx:strain-sobolev} are dual (adjoint) to each other when one of them has compact support.

\section{Main results}
\label{sec:main-result}
In this section, we summarize the main results of this paper. 

We first introduce the Dirac delta located on (sub)simplices. Given a simplex $\sigma$, we define the scalar delta $\delta_{\sigma}$ by 
$$\langle \delta_{\sigma}, u \rangle = \int_{\sigma} u,$$
and the vector delta $\delta_{\sigma}[\bm a]$ by 
$$\langle \delta_{\sigma}[\bm a], \bm v \rangle  = \int_{\sigma} \bm a \cdot \bm v,$$
and the tensor (matrix) delta $\delta_{\sigma}[\bm A]$ by 
$$\langle \delta_{\sigma}[\bm A], \bm W \rangle = \int_{\sigma} \bm A : \bm W,$$
for smooth function $u$, vector field $\bm v$, and tensor field $\bm W$, respectively. Here $\bm a\in \mathbb R^n$ and $\bm A\in \mathbb R^{n\times n}$ for $n = 2,3$. We say that $\delta_{\sigma}[\bm A]$ is symmetric (traceless) if $\bm A$ is symmetric (traceless). {We will often use $\delta_{\sigma}[\bm u \otimes \bm v]$, which by definition is $\bm W \mapsto \int_{\sigma} \bm W: (\bm u\otimes \bm v) = \int_{\sigma} \bm u\cdot \bm W \cdot \bm v$. It is symmetric if $\bm u = \bm v$, and it is traceless if $\bm u \perp \bm v$.}

\subsection{Motivation: Discrete Exterior Calculus (DEC) interpretations of the distributional de~Rham and Regge complexes}

We first provide some remarks for the de~Rham complex as a motivation.
The finite element de~Rham complex (Whitney forms, Figure \ref{fig:deRham-whitney}) can be viewed as a discretization for
\begin{equation}\label{cplx:3D-Whitney-sobolev}
\begin{tikzcd}
0 \arrow{r} & H^1 \arrow{r}{\grad} & H(\curl) \arrow{r}{\curl} & H(\div)\arrow{r}{\div}& L^2\arrow{r}&0,
\end{tikzcd}
\end{equation}
while the distributional complex (Figure \ref{fig:deRham-distributional}) is a slightly nonconforming discretization for 
\begin{equation}\label{cplx:3D-Whitney-sobolev0}
\begin{tikzcd}
0 \arrow{r} & L^2 \arrow{r}{\grad} & H^{-1}_0(\curl) \arrow{r}{\curl} & H_0^{-1}(\div)\arrow{r}{\div}& H_0^{-1}\arrow{r}&0,
\end{tikzcd}
\end{equation}
or equivalently, due to the duality \eqref{duality-sobolev},
\begin{equation}\label{cplx:3D-Whitney-sobolev-2}
\begin{tikzcd}
0 \arrow{r} & {[L^2]}^\ast \arrow{r}{\grad} & {[H(\div)]}^\ast \arrow{r}{\curl} & {[H(\curl)]}^\ast\arrow{r}{\div}& {[H^1]}^\ast\arrow{r}&0.
\end{tikzcd}
\end{equation}
The nonconformity is due to the fact that the canonical degrees of freedom (integrating $k$-forms on $k$-chains) is not well-defined for $L^2$ spaces.

We observe that the distributional de~Rham complex  (Figure \ref{fig:deRham-distributional}) is formally a reversed version of the standard de~Rham finite element complex, i.e., the $k$-th space is discretized on $(n-k)$-cells. In the Discrete Exterior Calculus (DEC) perspective, this means that $k$-forms are discretized on dual $k$-chains and the operators become discrete exterior derivatives on the dual mesh in the usual sense (see Figure \ref{fig:dec}). 
 This leads to another perspective of the distributional complexes: {\it DEC uses dual of meshes, while distributional finite elements use dual of spaces and operators}. This establishes a link between DEC and finite elements. In DEC, there are different ways to define dual meshes, which might affect the convergence; while for distributional finite elements, there is a canonical definition for the distributional derivatives.  Note that the dual of a simplicial mesh is not simplicial. Therefore, there are no trivial generalizations of Whitney forms (shape functions) on the dual mesh.  

\begin{figure}[htbp] 
   \centering
  \includegraphics[width=0.8\textwidth]{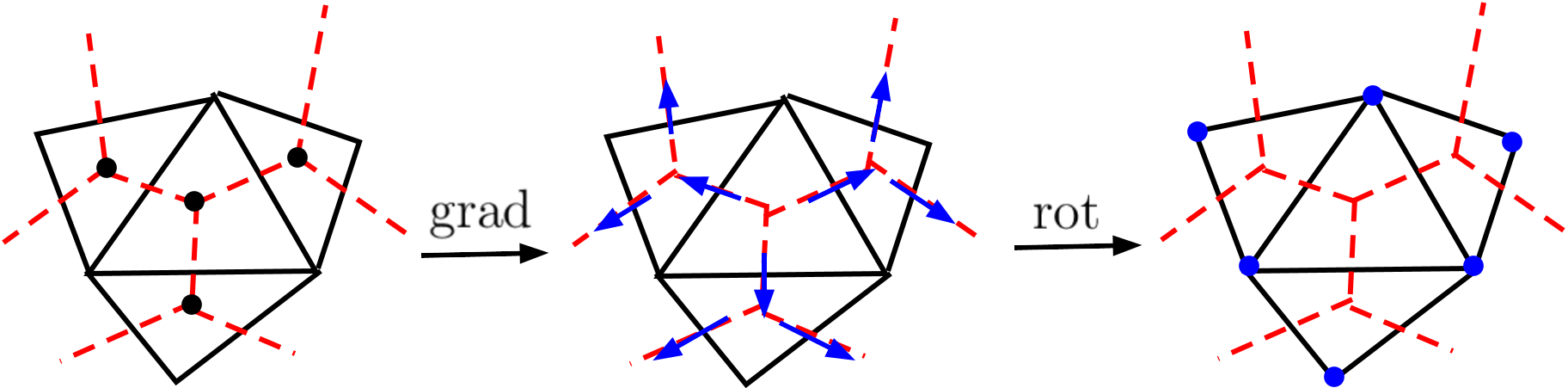} 
   \caption{Discrete Exterior Calculus perspective of a distributional de~Rham complex in 2D. The black cells form the primal mesh, and the red dashed lines form the dual mesh. The degrees of freedom and Dirac delta in the distributional complex correspond to DEC discrete forms on the dual mesh. Alternatively, one may view the red cells as primal and black as dual. Then the distributional spaces correspond to discrete forms on the primal mesh.}
   \label{fig:dec}
\end{figure}


To motivate the construction in this paper for other BGG complexes, we make some observations on the Regge complex (Figure \ref{fig:regge}) from a discrete form point of view. 
  We denote the Regge complex (Figure \ref{fig:regge}, \cite{christiansen2011linearization}) as follows (compared with \eqref{elasticity-Lambda}):
\begin{equation}\label{cplx:regge}
\begin{tikzcd}
0\arrow{r}&\Lambda_{h}^{0, 1} \arrow{r}{\deff} & \Lambda_{h}^{1, 1} \arrow{r}{\inc} &\Lambda_{h}^{2, 2} \arrow{r}{\div} &\Lambda_{h}^{3, 2} \arrow{r}{} &0.
\end{tikzcd}
\end{equation}
Here, $\Lambda_{h}^{0, 1}$ is a vector-valued version of the Lagrange finite element; $\Lambda_{h}^{1, 1}$ consists of Regge metrics, which are piecewise constant symmetric matrix fields, and the degrees of freedom are given by the lengths of edges $\int_{e}\bm t_{e}\cdot \bm g\cdot \bm t_{e}$; $\Lambda_{h}^{2, 2}$ is the dual of $\Lambda_{h}^{1, 1}$, which consists of tangential-tangential Dirac delta on the edges; and $\Lambda_{h}^{3, 2}$ is the dual of the Lagrange element (vertex evaluation). 
The Regge complex \eqref{cplx:regge} is a slightly nonconforming discretization of \eqref{cplx:elasticity-sobolev}, and is self-adjoint: 
 $\Lambda_{h}^{3, 2}$ is the dual of $\Lambda_{h}^{0, 1}$ and $\Lambda_{h}^{2, 2}$ is the dual of $\Lambda_{h}^{1, 1}$. This is consistent with the structure in \eqref{cplx:elasticity-sobolev}. 






On the continuous level, the sequence \eqref{elasticity-Lambda} starts with $\Lambda^{0, 1}$ (1-form-valued 0-form). In the Regge complex, this is discretized by a vector-valued Lagrange element and can be viewed as attaching an alternating 1-form at each 0-chain (recalling that in de~Rham sequences, one attaches a scalar value, or alternating 0-form, to each 0-chain). The discrete version of $\Lambda^{1, 1}$ in \eqref{elasticity-Lambda} is the Regge metric. The shape functions are $\bm t_{e}\odot \bm t_{e}, e\in \E$, and the degrees of freedom can be viewed as $\omega_{e}\odot \omega_{e}$, where $\omega_{e}$ is the dual 1-form of the constant vector field $\bm t_{e}$. This is a natural extension of discrete 1-forms in the de~Rham case. The third space, $\Lambda^{2, 2}$ seems rather different from the de~Rham case, as 1) we have distributions here, rather than piecewise functions; 2) the degrees of freedom are located on edges, rather than on faces as de Rham 2-forms in 3D are. Our interpretation of this difference is that, in the extended picture for \eqref{elasticity-Lambda}, following the zig-zag in the BGG diagram (which leads to the second-order operator $\inc$), we {\it move to the dual mesh}. Tangential-tangential components on the primal edges correspond to the normal-normal components on the dual mesh. Therefore, we interpret their space in the Regge complex (Figure \ref{fig:regge}) as $\tilde{\bm n}_{e}\odot \tilde{\bm n}_{e}$, where $\tilde{\bm n}_{e}= \bm t_{e}$ is the normal vector on the dual faces or the tangent vector on the primal edges. This thus has a neat correspondence to the de~Rham 2-forms in 3D, which are discretized in the normal direction of faces. Moreover, we can further interpret $\tilde{\bm n}_{e}\odot \tilde{\bm n}_{e}$ as $(\tilde{\bm t}_{e}^{1}\wedge \tilde{\bm t}_{e}^{2})\odot (\tilde{\bm t}_{e}^{1}\wedge \tilde{\bm t}_{e}^{2})$, where $\tilde{\bm t}_{e}^{1}$ and $\tilde{\bm t}_{e}^{2}$ are the two tangent vectors of the dual face of $e$, and $[\tilde{\bm t}_{e}^{1}, \tilde{\bm t}_{e}^{2}, \tilde{\bm n}_{e}]$ forms an orthonormal frame on the face dual to $e$. This interpretation allows us to view the third space in the Regge complex as a 2-form-valued 2-forms, or a 4-tensor which is skew-symmetric with respect to the first two legs and with respect to the last two legs, with the additional symmetry that we can exchange the first two indices with the last two. This is exactly the symmetry of the Riemannian tensor.  In fact, this corresponds to the identification of the Riemannian tensor with the Ricci or Einstein tensor in 3D.

\subsection{Hessian complex in 3D}
Inspired by these observations, we establish Hessian and divdiv complexes. For the Hessian complex, the Lagrange element is a natural discretization for $\Lambda^{0, 0}$ as it assigns a scalar value to each vertex (0-chain). In the Regge complex, we see that one turns to the dual mesh following the zig-zag ($\inc=\curl \circ \mathcal S^{-1}\circ \curl$). Correspondingly, in the Hessian complex, a natural guess is that $\hess=\grad\circ I \circ \grad$ takes us to the dual mesh, leading to discretizing $\Lambda^{1, 1}$ on the tangential direction of dual 1-chain, i.e., the normal direction of the primal 2-chains.  In fact, we verify below that this is the case. As a summary, our discrete Hessian complex is the following (see Figure \ref{fig:3D-hessian}): 
\begin{equation}\label{cplx:hessian-3d}
\begin{tikzcd}
0\arrow{r}&V^0 \arrow{r}{\hess} & \bm V^1 \arrow{r}{\curl} &\bm V^2\arrow{r}{\div} &\bm V^3 \arrow{r}{} &0.
\end{tikzcd}
\end{equation}
 {This is a discrete version of \eqref{cplx:hessian-sobolev}.}

\begin{figure}[htbp] 
   \centering
  \includegraphics[width=0.8\textwidth]{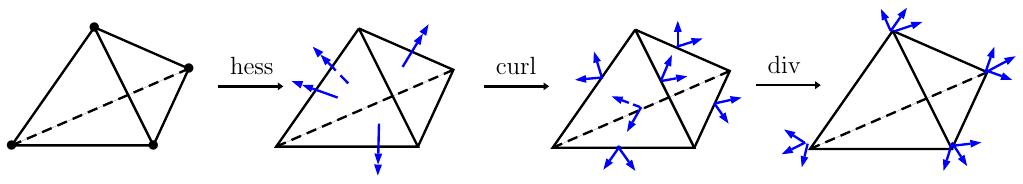} 
   \caption{Hessian complex in 3D \eqref{cplx:hessian-3d}.}
   \label{fig:3D-hessian}
\end{figure}

Now we construct the spaces $V^0$, $\bm V^1$, $\bm V^2$, and $\bm V^3$ using vector and tensor deltas.

The space $V^0$ is chosen as the standard Lagrange space, namely,
\begin{equation}
V^0 := \{ u \in C^0(\Omega) : u|_{K} \in \mathcal P_1(K), \text{ for each cell } K \in \mathsf K\}.
\end{equation}
The space $\bm V^1$ consists of normal-normal deltas on faces:
$$\bm V^1 := \Span\{ \delta_{f}[\bm n_{f} \otimes \bm n_{f}] ; f \in \mathsf F_0\},$$
where $\bm n_{f}$ is the normal vector of $f$. Clearly, the dimension of $\bm V^1$ is $\dim \bm V^1 = \sharp_{\mathsf F_0}.$
The space $\bm V^2$ consists of normal-tangential deltas on edges:

\begin{equation}
\bm V^2 := \Span\{ \delta_{e}[\bm m_{e}\otimes \bm t_{e}] :~ \bm m_{e} \text{ is normal to } e, ~e \in \mathsf E_{0}\}.
\end{equation}

Note that for each internal edge $e$, the normal vector (with respect to $e$) forms a two-dimensional space. We choose a basis $\bm n_{e, +}$ and $\bm n_{e,-}$ of this space. The space $\bm V^2$ can be equivalently reformulated as 
\begin{equation}
    \bm V^2 = \Span\{ \delta_{e}[\bm n_{e,+}\otimes \bm t_{e}], \,  \delta_{e}[\bm n_{e,-}\otimes \bm t_{e} ] :~e \in \mathsf E_{0}\}.
\end{equation}

As a consequence, the dimension of $\bm V^2$ is $\dim \bm V^2 = 2\sharp_{\mathsf E_0}$.

Finally,  $\bm V^3$ consists of the vector-valued delta at interior vertices: 
$$\bm V^3 := \Span\{ \delta_{ x}[\bm a_{ x}], \bm a_{ x} \in \mathbb R^3: x\in \mathsf V_{0}\}.$$

We have explicit forms of the differential operators on these function or distributional spaces, which we will verify using definitions of distributional derivatives. 

We can identify $u$ with $\sum_{K \in \mathsf K} u|_{K} \|K\|$. For $u\in V^{0}$, 
\begin{equation}
\hess u = \sum_{f \in \mathsf F_0} \delta_{f} \big[ [\![\nabla u ]\!]_{f} \otimes \bm n_{f}\big].
\end{equation}
Here $[\![\nabla u ]\!]_{f} = \sum_{K : K \supset f} \mathcal O(f,K) \nabla u|_K$ is the jump of $\nabla u$ on $f$, which has vanishing tangential components. 
For $ \bm \sigma = \sum_{f \in \mathsf f_0} \bm \delta_{f}[\bm m_{f} \otimes \bm n_{f}] \in \bm V^{1}$ with $\bm m_f$ normal to $f$, 
we have
    \begin{equation}
        \label{eq:curlV1}
        \curl \bm \sigma = \sum_{e \in \mathsf E_0} \sum_{f:f \supset e} \mathcal O(e, f) \delta_{e}[\bm m_{f} \otimes \bm t_{e}].
    \end{equation}
   The condition $f \supset e$ implies that $\bm m_{f}$ is normal to the edge $e$. Consequently, we have $\curl \bm V^1 \subseteq \bm V^2$.
    Let $\bm v = \sum_{e \in \mathsf E_0} \bm \delta_{e}[\bm m_{e} \otimes \bm t_{e}]\in \bm V^2$,
    with $\bm m_{e}$ normal to the edge $e$. 
    Then, it holds that 
    \begin{equation}\label{eq:div-formula}
        \div \bm v = -\sum_{x \in \mathsf V_0} \sum_{e:e \ni x} \mathcal O(x, e)\delta_{x}[\bm m_{e}] .
    \end{equation}
    This implies that, $\div \bm V^2 \subseteq \bm V^3.$

\begin{theorem}
    \label{thm:hom-hessian-3d}
The sequence \eqref{cplx:hessian-3d} is a complex, and its cohomology is isomorphic to the $\mathcal{P}_{1}$-valued de Rham cohomology $ \mathcal{H}^{\bs}_{dR}(\Omega)\otimes \mathcal{P}_{1}$.
\end{theorem}

We now give a geometric interpretation of the distribution spaces defined above. The space $\bm V^1$ can be regarded as a  direct sum of $\mathcal P_0^n(f) \|f\|$, where $\mathcal P_0^n(f) := \{ \bm a \in \mathbb R^3 : \bm a \text{ is normal to } f\}.$ Similarly, we can define $\mathcal P_0^n(e)$ as $\mathcal P_0^n(e) := \{\bm a \in \mathbb R^3 : \bm a \text{ is normal to }e\}$, and $\mathcal P_0^n(x) = \mathcal P_0 \otimes \mathbb R^3.$ 
Therefore, there is an isomorphism from the distribution spaces $\bm V^i$ and the direct sum of $\mathcal P_0^n$'s. 

\begin{equation}\label{discrete-geo-hess}
\begin{tikzcd}
0 \ar[r] 
&
  \bm V^1 \ar[r]  \ar[d]
  & 
  \bm V^2 \ar[r] \ar[d]
  & 
  \bm V^3 \ar[r] \ar[d]
  & 
  0 
  \\
0 \ar[r] 
& 
\bigoplus_{f \in \mathsf F_0} \mathcal P_0^n(f) \ar[r,"\partial"] 
&  
\bigoplus_{e\in \mathsf E_0} \mathcal P_0^n(e) \ar[r,"\partial"] 
&
\bigoplus_{x \in \mathsf V_0} \mathcal P_0^n(x) \ar[r] 
& 
0.
\end{tikzcd}
\end{equation}
Here the boundary operators in the bottom row are the standard boundary operators (see Section \ref{sec:topology} and \cite{arnold2018finite,hatcher2002algebraic}). For the bottom row of \eqref{discrete-geo-hess},  we note the fact $\mathcal P_0^n(f) \subset \mathcal P_0^n(e)$ and identify $\mathcal P_0^n(f)$ as subspaces of $\mathcal P_0^n(e)$ in the definition of the boundary operators, see \Cref{rmk:subspace-diff}. 

We also have another version of the Hessian complex with homogeneous boundary conditions:
\begin{equation}\label{cplx:hessian0-3d}
\begin{tikzcd}
0\arrow{r}&V^0_{0} \arrow{r}{\hess_0} & \bm V^1_{0} \arrow{r}{\curl_0} &\bm V^2_{0}\arrow{r}{\div_0} &\bm V^3_{0} \arrow{r}{} &0.
\end{tikzcd}
\end{equation}
 {
Here the differential operators call for more explanations. Note that \eqref{cplx:hessian0-3d} is a discrete version of \eqref{eq:hessian_lowreg-exact}, and the distributions in \eqref{eq:hessian_lowreg-exact} are defined as a space of distributions in $\mathbb R^n$. Therefore, the differential operators are defined with test functions from $C_c^{\infty}(\mathbb R^n)$.  
To emphasize the difference, we will use the subscript $0$ to represent the differential operators in $H_0^s = H_{\overline{\Omega}}^s$ sequence. More specifically, we define 
$$\langle \hess_0 u, \bm \sigma \rangle = \langle u, \div\div \bm \sigma \rangle, \quad \forall \bm \sigma \in C_c^{\infty}(\mathbb R^3; \mathbb S), $$
and similarly we define $\curl_0$ and $\div_0$ using test functions in $C_c^{\infty}(\mathbb R^3; \mathbb T)$ and $C_c^{\infty}(\mathbb R^3; \mathbb R^3)$.

A significant difference is that functions in $C_c^{\infty}(\Omega)$ have vanishing traces in $\partial \Omega$ while the traces of functions in $C_c^{\infty}(\mathbb R^3)$ are not necessarily zero. This implies that the operators $\hess_0$, $\curl_0$ and $\div_0$ might lead to boundary terms, which disappear in the spaces in \eqref{cplx:hessian-3d}. 

The observation above motivates us to define $\bm V_0^1$, $\bm V_0^2$, and $\bm V_0^3$ via modifications of boundary terms. Let $$\bm V^1_0 := \Span\{ \delta_{f}[\bm n_{f} \otimes \bm n_{f}]: f \in \mathsf F\},$$
$$
    \bm V^2_0 := \Span\{ \delta_{e}[\bm m_{e} \otimes \bm t_{e}] :~ \bm m_{e} \text{ is normal to } e, ~e \in \mathsf E\},
$$
and $$\bm V^3_0 := \Span\{ \delta_{ x}[\bm a_{ x}]: \bm a_{ x} \in \mathbb R^3, x\in \mathsf V\}.$$
Finally, $V_0^0$ is the piecewise linear Lagrange element space with homogeneous boundary conditions. 
\begin{theorem}
    \label{thm:hom-hessian0-3d}
The sequence \eqref{cplx:hessian0-3d} is a complex, and its cohomology is isomorphic to $ \mathcal{H}^{\bs}_{dR,c}(\Omega)\otimes \mathcal{P}_{1}$.
\end{theorem}

}

\subsection{Divdiv complex in 3D}
\label{sec:main-divdiv3d}

The $\div\div$ complex (see Figure \ref{fig:3D-divdiv}) is inspired by the dual of the Hessian complex \eqref{cplx:hessian0-3d}, which we denote as 
\begin{equation}\label{cplx:divdiv-3d}
\begin{tikzcd}[column sep=large]
0 \arrow{r}& \bm U^0 \arrow{r}{\dev\grad}& \bm U^1 \arrow{r}{\sym\curl_{h}} &\bm U^2 \arrow{r}{\widehat{\div\div}} & U^3\arrow{r}& 0.
\end{tikzcd}
\end{equation}
\begin{figure}[htbp] 
   \centering
  \includegraphics[width=0.8\textwidth]{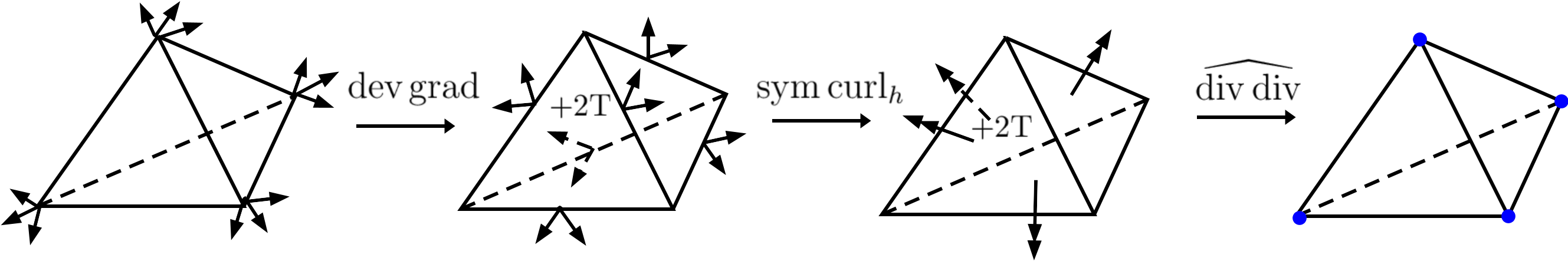} 
   \caption{Divdiv complex in 3D \eqref{cplx:divdiv-3d}.}
   \label{fig:3D-divdiv}
\end{figure}
The space $U^3$ is defined as the dual of $V^0_{0}$. Note that $V^0_{0}$ is the Lagrange finite element space with homogeneous boundary conditions, and therefore $U^3$ consists of Dirac delta at interior vertices.
 Clearly, there is a dual pairing between $ U^3$ and $V^0_{0}$: For given function $v \in V^0_{0}$ and $\hat{ v} = \sum_{ x \in \mathsf V} u_{ x} \delta_{ x} \in U^3$, we set
\begin{equation}
\langle v, \hat{v} \rangle_{V^0_{0}\times U^3} = \sum_{ x \in \mathsf V} u_{ x} v( x).
\end{equation}
The space $\bm U^0$ is defined as the dual of $\bm V_{0}^3$. Recall that $\bm V_{0}^3$ is the (vector) vertex delta function. As a consequence, we choose $\bm U^0$ to be the standard Lagrange space. 

The space $\bm U^2$ is inspired by the dual of $\bm V^1_{0}$. Recall that $\bm V^1_{0}$ consists of face normal-normal delta (including the boundary modes).
This motivates us to choose $\bm U^2$ as a finite element function space with normal-normal continuity. A natural choice is the stress element of the TDNNS formulation (see \cite{pechstein2011tangential}): 
\begin{equation}
\bm U^2 = \{ \bm \sigma \in L^2(\Omega; \mathbb S) : \bm \sigma|_{K} \in \mathcal P_0(\mathbb S), K\in \mathsf K, \bm n_f \cdot \bm \sigma \cdot \bm n_f \text{ is continuous across faces}\}.
\end{equation}
The shape function space of $\bm U^2$ is constant symmetric matrices, which has six dimensions. In the context of the divdiv complex, $\bm U^2$ and $U^3$ are related to the Hellan-Herrmann-Johnson (HHJ) formulation for the biharmonic problem \cite{chen2023new,li2018regge,neunteufel2023hellan}.

The degrees of freedom are given by the evaluation of the normal-normal component on each face, i.e., $\int_{f}\bm n_f \cdot \bm \sigma \cdot \bm n_f$, for any $f\in \mathsf F$, plus interior degrees of freedom.  More precisely, 
we first decompose constant symmetric matrices as follows:
\begin{equation}
    \mathbb S = \mathbb B_{K} \oplus \mathcal I_{ K},
\end{equation}
where
$\mathbb B_{K}$ is the normal-normal bubble, namely, 
\begin{equation}
\mathbb B_{K} = \{ \bm \sigma \in \mathbb S : \bm n_{ f} \cdot \bm \sigma \cdot \bm n_{ f}=0 \text{ for each face }  f \text{ of }  K\},
\end{equation}
and the $\mathcal I_{ K}$ is  the orthogonal complement of $\mathcal B_{K}$ with respect to the Frobenius inner product. 
The degree of freedom of $\bm U^2$ can be given as the following: for a given $\sigma \in \mathbb S$, 
\begin{enumerate}[label = (1\alph*)]
\item\label{dof:divdiv1} The normal-normal component on each face $f$: $\int_{f}\bm n_{ f} \cdot \bm \sigma \cdot \bm n_{ f}$;
\item\label{dof:divdiv2} The inner product against the normal-normal bubble $\int_{K} \bm \sigma : \bm \tau$ for $\bm \tau \in \mathbb B_{K}$.
\end{enumerate}
The dimension of the global finite element space is thus $\sharp_{F}+2\sharp_{K}$.

Similarly, $\bm U^1$ is inspired by the dual of $\bm V_{0}^2$. As we have enriched $\bm U^2$ with two interior degrees of freedom for unisolvency, we should also enrich the local shape function space of $\bm U^1$. More precisely, we note that the following is an exact sequence:
\begin{equation}\label{seq:polynomial-divdiv}
\begin{tikzcd}[column sep=large]
\cdots \arrow{r}&\mathcal{P}_{1}(\mathbb{V})\arrow{r}{\dev\grad}& \mathcal{P}_{0}(\mathbb{T})+ \bm x\times \mathcal{P}_{0}(\mathbb{S})   \arrow{r}{\sym\curl} &  \mathcal{P}_{0}(\mathbb{S})\arrow{r}&\cdots.
\end{tikzcd}
\end{equation}
The exactness follows from the existence of null-homotopy operators \cite{vcap2023bounded}, as
$$
 \mathcal{P}_{0}(\mathbb{T})+ \bm x\times \mathcal{P}_{0}(\mathbb{S})=\dev\grad \mathcal{P}_{1}(\mathbb{V}) \oplus \kappa \mathcal{P}_{0}(\mathbb{S}),
$$ 
where $\kappa$ is the Poincar\'e operator from \cite{vcap2023bounded}. Note that $\bm x\times \mathbb{S}$ is automatically trace-free (the same algebraic structure as the fact that $\curl$ of a symmetric matrix is trace-free). 
\begin{remark}
To show the exactness of \eqref{seq:polynomial-divdiv}, it suffices to follow the construction in \cite{vcap2023bounded} with smooth functions. The formulas in \cite{vcap2023bounded} involve derivative terms (like the Ces\`aro-Volterra formula), but these terms vanish on constants. The exactness of the above complexes can be also checked with direct computation, see \cite{chen2022finitedivdiv}.
\end{remark}

The degrees of freedom for $\bm U^1$ on $K$ are given by, for any $\bm \sigma\in \bm U^1$,
\begin{enumerate}
[label = (2\alph*)]
\item\label{dof:symcurl1}  $\int_{e}\bm n_{e, + }\cdot \bm \sigma\cdot\bm t_{e} , \,\, \int_{e}\bm n_{e, -}\cdot \bm \sigma\cdot \bm t_{e}, \quad\forall e\in \mathsf{E}(K)$,
\item\label{dof:symcurl2}  $\int_{K}\curl \bm \sigma :\bm b, \quad \forall \bm b\in \mathbb B_{K}$.
\end{enumerate}
\begin{theorem}
\label{thm:unisolvency-U1}
The above degrees of freedom are unisolvent for  $\bm U^1$. 
\end{theorem}
Traceless finite elements with tangential-normal continuity have been used in Mass Conserving mixed Stress (MCS) formulations for the Stokes equations \cite{gopalakrishnan2020mass}. Here $U^1$ is different from the construction in \cite[Section 5.1]{gopalakrishnan2020mass}. To see this, it suffices to note that in the lowest order case, the finite element in \cite[Section 5.1]{gopalakrishnan2020mass} has dimension 16 (see  \cite[Theorem 5.4]{gopalakrishnan2020mass}). Nevertheless, $U^1$ has dimension 14. Moreover,  the local shape function space of $U^1$ has a construction based on the Poincar\'e map $\bs\times x$. This is important for characterising the kernel of $\sym\curl_h$ (as a special case of the cohomology result).


Next, we explain the operators in \eqref{cplx:divdiv-3d}. The $\dev\grad$ operator is a canonical one. The operator $\sym\curl_{h}$ means the piecewise $\sym\curl$ operator.  Finally, as $\bm U^2$ is a piecewise function, $\div\div$ in the sense of distributions will map $\bm U^2$ to Dirac deltas on edges (codimension two cells). However, we introduce a new operator $\widehat{\div\div}$, which is defined with a similar idea as the TDNNS method \cite{pechstein2011tangential}. More precisely, the first $\div$ maps a normal-normal continuous matrix field to the edge deltas (dual of the N\'ed\'elec element) as in TDNNS. Then the second $\div$ has the usual definition in the sense of distributions. Alternatively, one may understand $\widehat{\div\div}$ as the distributional $\div\div$ restricted to vertices.  The rigorous definition is provided in \eqref{eq:hat-divdiv-3d}. 
\begin{theorem}
\label{thm:hom-divdiv-3d}
The sequence \eqref{cplx:divdiv-3d} is a complex, and the cohomology is isomorphic to $\mathcal{H}_{dR}^{\bs}\otimes \mathcal{RT}$. 
\end{theorem}
The two spaces in the middle of  \eqref{cplx:divdiv-3d} contain interior degrees of freedom. To obtain a precise dual version of \eqref{cplx:hessian-3d} with a neat discrete form interpretation, we can eliminate the interior degrees of freedom from the two spaces simultaneously. {Let $\widehat{\bm U}^1$ and $\widehat{\bm U}^2$ be the space spanned by the dual basis of \ref{dof:divdiv1} and \ref{dof:divdiv2}, respectively. }This leads to 
 \begin{equation}\label{cplx:divdiv-}
\begin{tikzcd}[column sep=large]
0 \arrow{r}& \bm U^0 \arrow{r}{\dev\grad}&  {\widehat{ \bm U}}^1 \arrow{r}{\sym\curl_{h}} & \widehat{\bm U}^2 \arrow{r}{\widehat{\div\div}} & U^3\arrow{r}& 0.
\end{tikzcd}
\end{equation}
Here  $\widehat{\div\div}$ is defined as \begin{equation}
    \label{eq:hat-divdiv-3d}
    \langle \widehat{\div \div} \bm \sigma, v \rangle  :=  \sum_{f \in \mathsf F} \sum_{K: K \supset f}\int_{f} \Big[\mathcal O(f,K)\bm n_f  \times \bm \sigma|_K \cdot \bm n_f \Big] \cdot (\bm n_f\times\nabla v ).
    \end{equation}
The definition comes from the following observation. For the distributional $\div\div$ operator, it follows that for smooth $\bm u$,
\begin{equation}
\begin{split}
\langle \div \div \bm \sigma, \bm u \rangle := \langle \bm \sigma, \hess u \rangle = & \sum_{K \in \mathsf K} \int_{K}  \bm\sigma : \hess u \\
=& \sum_{K \in \mathsf K} \int_{\partial K} \bm \sigma\bm n \cdot \nabla u \\ 
= & \sum_{f \in \mathsf F_0} \sum_{K: K \supset f} \int_{f} \Big[ \mathcal O(f,K) \bm \sigma|_K \cdot \bm n_f \Big] \cdot \nabla u \\ 
= &\sum_{f \in \mathsf F_0} \sum_{K: K \supset f}\int_{f} \Big[\mathcal O(f,K) \bm n_f \times  \bm \sigma|_K \cdot \bm n_f \Big] \cdot ( \bm n_f\times \nabla  u ).
\end{split}
\end{equation}
The right-hand side makes sense for continuous, piecewise smooth function. Restricting the distributional $\div\div$ operator to the Lagrange space $\bm U^3$ leads to the $\widehat{\div\div}$ operator. 

The trimmed space $  \widehat{\bm U}^2$ gains a precise duality to $\bm V^1_{0}$ and a canonical form interpretation (analogous to the Whitney forms) by sacrificing some approximation property since it does not contain piecewise constant. But further applications and analysis, e.g., in solving PDEs, are beyond the scope of this paper. 


We also propose a compactly supported version of the $\div\div$ complex as a dual of \eqref{cplx:hessian-3d},
\begin{equation}\label{cplx:divdiv0-3d}
    \begin{tikzcd}[column sep=large]
    0 \arrow{r}& \bm U^0_0 \arrow{r}{\dev\grad}& \bm U^1_0 \arrow{r}{\sym\curl_{h}} &\bm U^2_0 \arrow{r}{\widehat{\div\div}_0} & U^3_0\arrow{r}& 0.
    \end{tikzcd}
    \end{equation}
Here $U^3_0$ are spanned by vertex deltas of all the vertices $x \in \mathsf V$. Finally, $\bm U^i_0$ are the subspace of $\bm U^1$ consisting of functions  such that the degrees of freedom on the boundary vanish.

\begin{theorem}
    \label{thm:hom-divdiv0-3d}
    The sequence \eqref{cplx:divdiv0-3d} is a complex, and the cohomology is isomorphic to $\mathcal{H}_{dR,c}^{\bs}\otimes \mathcal{RT}$.
    \end{theorem}
    
    \subsection{Complexes in 2D}

We construct discrete versions of the Hessian and divdiv complexes and show that their cohomology is isomorphic to the continuous version. 

With a slight abuse of notation, we adopt the same notation as the 3D version and denote the 2D discrete Hessian complex as
\begin{equation}\label{cplx:hessian-2d}
\begin{tikzcd}[column sep=large]
0 \arrow{r}  &V^{0} \arrow{r}{\hess} &\bm V^{1} \arrow{r}{\rot}&\bm V^{2} \arrow{r}{} &  0.
 \end{tikzcd}
\end{equation}
  \begin{figure}[htbp]\centering
\includegraphics[width=0.6\linewidth]{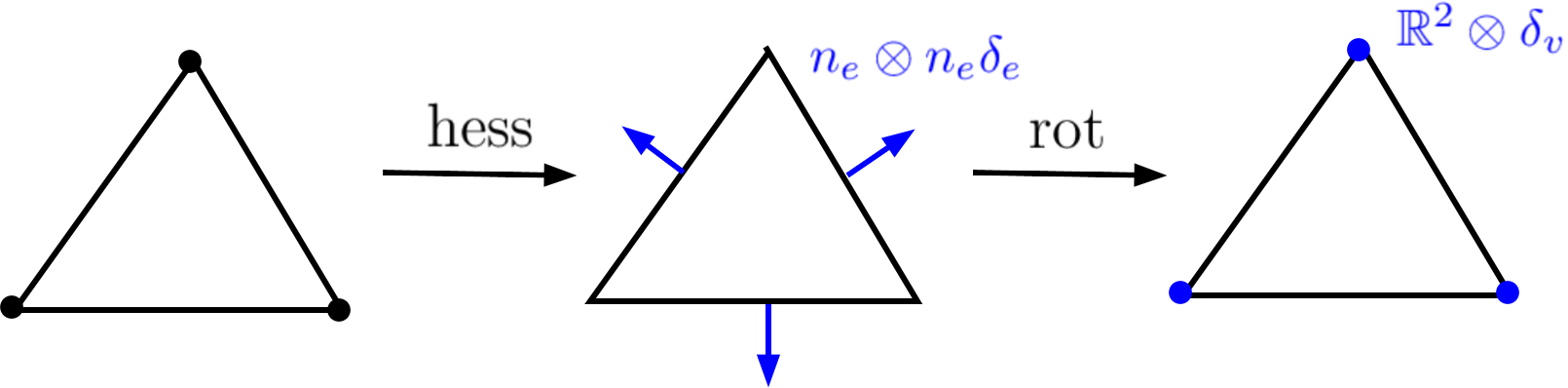}
\caption{Distributional Hessian complex in 2D \eqref{cplx:hessian-2d}. }
\label{fig:2D-hess} 
\end{figure}
  \begin{figure}[htbp]
\includegraphics[width=0.6\linewidth]{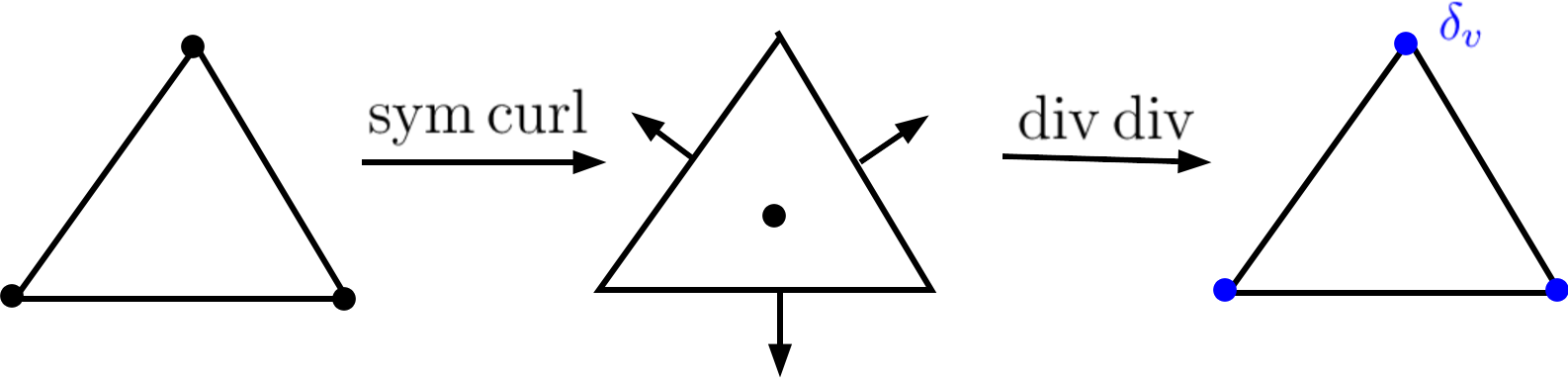} 
\caption{Distributional divdiv complex in 2D \eqref{cplx:divdiv-2d}. }
\label{fig:2D-divdiv} 
\end{figure}
The space $V^{0}$ is the first-order Lagrange finite element space consisting of continuous piecewise linear functions. 
 The space $\bm V^{1}$ consists of normal-normal deltas on each interior edge, i.e., 
$$
\bm V^{1}:=\left\{\sum_{e\in \mathsf E_0}u_{e}\delta_e[\bm n_e\bm n_e^{T}] : u_e\in \mathbb R \right\}.
$$
The space $\bm U^{2}$ consists of Dirac deltas at each interior vertex:
$$
\bm V^{2}:=\left\{\sum_{x\in \mathsf V_0}\delta_{x}[\bm a_{x}] :\bm a_{x}\in \mathbb R^2 \right\}.
$$
\begin{theorem}\label{lem:closedness-hessian-2d}
    The sequence \eqref{cplx:hessian-2d} is a complex. The cohomology of \eqref{cplx:hessian-2d} is isomorphic to $ \mathcal H^{\bs}_{dR}(\Omega)\otimes \mathcal P_1$.
\end{theorem}

Similarly, we can introduce the version of homogeneous boundary conditions:
\begin{equation}
  \label{cplx:hessian0-2d}
\begin{tikzcd}
    0 \ar[r] & V^0_0 \ar[r,"\hess_0"] &  \bm V^1_0 \ar[r,"\rot_0"]& \bm V^2_0 \ar[r] &  0.
\end{tikzcd}
    \end{equation}
 The first space $V^0_0$ is the piecewise linear Lagrange finite element space with zero boundary condition, and correspondingly, 
$$
  \bm V^1_0 := \Big\{\sum_{ e \in \mathsf E} u_{e} \delta_{e}[\bm n_{e} \otimes \bm n_{e}] \Big\},
\text{  and  }
\bm V^2_0 := \Big\{ \sum_{x \in \mathsf V} \delta_{x}[\bm a_x] \Big\}.
$$
 Compared to $\bm V^1$ and $\bm V^2$, the only difference here is that the distribution spaces $ \bm V^1_0$ and $\bm V^2_0$ consist of Dirac deltas on the boundary. The cohomology of \eqref{cplx:hessian0-2d} is isomorphic to the compactly supported version $ \mathcal H^{\bs}_{dR, c}(\Omega)\otimes \mathcal P_1$. 
 
 The discrete divdiv complex reads as follows:
\begin{equation}
    \label{cplx:divdiv-2d}
\begin{tikzcd}[column sep=large]
0 \ar[r] & \bm U^0 \ar[r,"\sym\curl"] & \bm U^1 \ar[r, "\div\div"]  &U^2 \ar[r] &0.
\end{tikzcd}
\end{equation}
Here, 
the space $\bm U^0$ is the vector Lagrange finite element space
\begin{equation}
    \bm U^0 := \{ \bm u \in C^0(\Omega; \mathbb R^2) : \bm u |_{f} \in \mathcal P_1 \otimes \mathbb R^2 \text{ for each face } \bm f \in \mathsf F\},
\end{equation}
the space $\bm U^1$ is the rotated Regge element space:
\begin{equation}\label{divdiv-2d-U1}
\bm U^1 := \{ \bm \sigma \in L^2(\Omega; \mathbb S^{2\times2}) : \bm \sigma|_{f} \in \mathbb S^{2\times 2}, \bm n_{e} \cdot \bm \sigma \cdot \bm n_{e} \text{ is continuous on each edge}.\}
\end{equation}
The last space $ U^2$ is the scalar vertex delta : 
\begin{equation}
U^2 := \operatorname{span}\{ \delta_{x} : x \in \mathsf V_0\}.
\end{equation}
Similarly, we obtain the compactly supported version by alternating the boundary conditions. 
\begin{theorem}\label{thm:hom-divdiv-2d}
The cohomology of \eqref{cplx:divdiv-2d} is isomorphic to the de~Rham cohomology $ \mathcal H^{\bs}_{dR}(\Omega)\otimes \mathcal{RT}$, while the cohomology of the compactly supported version is isomorphic to $ \mathcal H^{\bs}_{dR, c}(\Omega)\otimes \mathcal{RT}$. 
\end{theorem}
We remark that  \eqref{cplx:divdiv-2d} (with $U^2$ replaced by its dual, the Lagrange finite element space) has appeared in \cite{chen2018multigrid}. Nevertheless, to the best of our knowledge, the cohomology was open.


\subsection{Tensor product construction}

Although in this paper we mostly focus on simplicial meshes and discrete structures on them,  we provide some brief remarks on a tensor product construction in this section. 

For the diagram \eqref{diagram-nD} and thus the complexes derived from it, there is a canonical construction on cubical grids \cite{bonizzoni2023discrete}. This was done by a tensor product of de~Rham and BGG complexes in 1D. The focus of \cite{bonizzoni2023discrete} was on conforming finite elements. Therefore, both the first row and the first {\it column} of \eqref{diagram-nD} were discretized by a standard conforming finite element sequence. However, much of the algebraic structure does not rely on this conformity and can be generalized once we have other patterns in 1D. In particular, we can start from distributional complexes in 1D and derive $n$D versions with a tensor product construction. The generalization is relatively straightforward. Therefore, we will not present the construction in full detail, but rather refer to Figure \ref{fig:1D-BGG}-\ref{fig:2D-tensor-product} for an illustration of the idea.
 \begin{figure} [htbp]
	\begin{minipage}[b]{0.4\textwidth}
    \includegraphics[width=2.in]{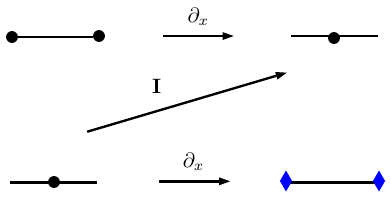} 
 \end{minipage}\hspace{+1.6cm}
 	\begin{minipage}[b]{0.4\textwidth}
    \includegraphics[width=2.in]{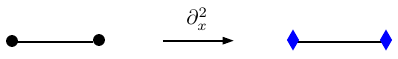} 
 \end{minipage}
    \caption{Discrete BGG diagram and complex in 1D. In the diagram (figure on the left), the top row is the canonical finite element de~Rham complex consisting of the Lagrange element (continuous piecewise linear) and piecewise constant. The bottom row starts with piecewise constants and ends up with Dirac delta at interior vertices. It is straightforward to check that the cohomology of both rows is isomorphic to the continuous version. Connecting the two rows, we derive the BGG (1D Hessian) complex (figure on the right). }
   \label{fig:1D-BGG}
  \end{figure}
\begin{figure}[H]
   \centering
  \includegraphics[width=0.5\textwidth]{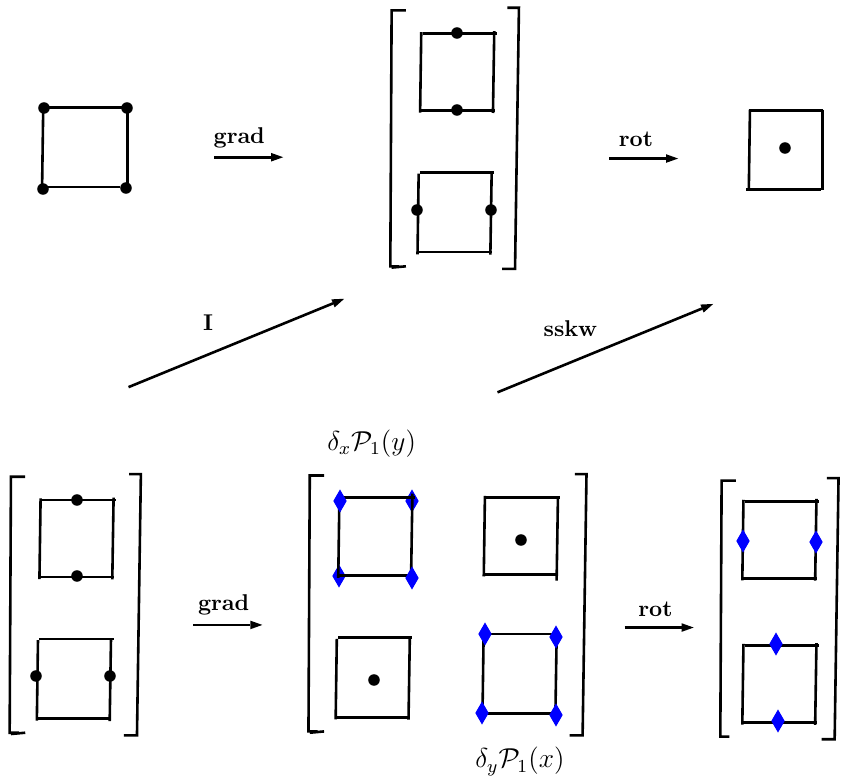} 
   \caption{BGG diagram for deriving the Hessian complex in 2D. The top row is a canonical finite element de Rham complex (see, e.g., \cite{buffa2011isogeometric,christiansen2009foundations}) obtained by a tensor product of the first row in Figure \ref{fig:1D-BGG} (left). The bottom row starts from a tensor product of constants and linear functions (the first space of the top and bottom rows in the diagram in Figure \ref{fig:1D-BGG}).}
   \label{fig:2D-tensor-product-bgg}
\end{figure}
\begin{figure}[htbp]
   \centering
  \includegraphics[width=0.5\textwidth]{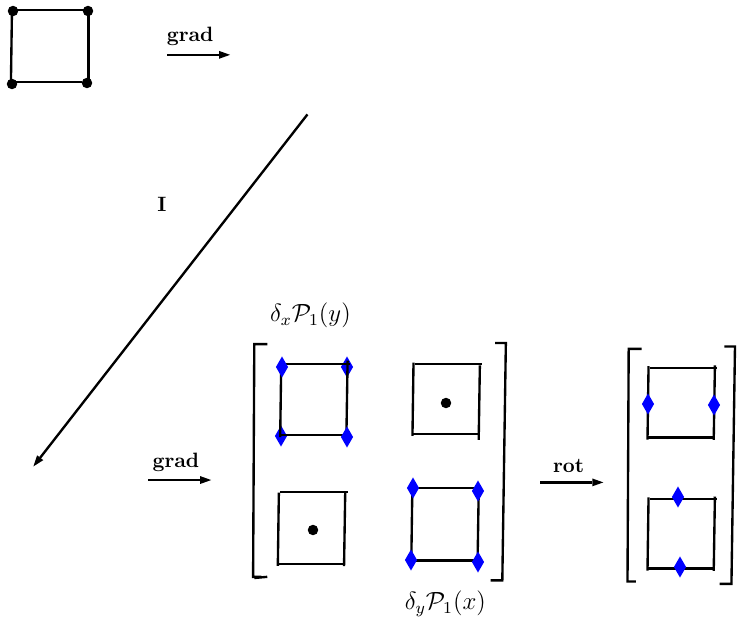} 
   \caption{Tensor product distributional Hessian complex in 2D derived from the diagram in Figure \ref{fig:2D-tensor-product-bgg}, which also follows from a generalization of \cite{bonizzoni2023discrete}.}
   \label{fig:2D-tensor-product}
\end{figure}

\section{Complexes in 2D}
\label{sec:2d}

To show the ideas of deriving cohomology in a relativity simple case, we start with complexes in 2D.

\subsection{2D Hessian complex}
In this section, we focus on the Hessian complex in 2D \eqref{cplx:hessian-2d}, showing that the sequence is a complex (\Cref{lem:closedness-hessian-2d}, part I) and  the cohomology is isomorphic to $ \mathcal H^{\bs}_{dR}(\Omega)\otimes \mathcal P_1$ (\Cref{lem:closedness-hessian-2d}, part II).

\begin{proof}[Proof of \Cref{lem:closedness-hessian-2d} (part I)]
{We identify $u$ with $ \sum_{f} u|_f\|f\|.$} 
For $u \in V^0$ and $\varphi\in C_c^{\infty}(\Omega; \mathbb R^{2\times 2})$, it holds that
\begin{align*}
	&\langle\grad\grad u,\varphi\rangle=\langle u,\div\div\varphi\rangle =\sum_{ f\in \mathsf F}\int_f u \div\div \varphi\\
	=&\sum_{f\in \mathsf F}\int_f \grad\grad u:\varphi-\int_{\partial f} \grad u \cdot \varphi\bm n =
  {-} \sum_{e\in \mathsf E_0} \sum_{f : f \supset e} \mathcal O(e,f) \int_e \nabla u|_f \cdot \varphi\bm n_e
\end{align*}
which implies 
\begin{equation}\label{eq:dist-gradgrad-2d}
\grad\grad u= {-}\sum_{e\in \mathsf E_0} 
\sum_{f \in \mathsf F} \mathcal O(e,f) (\nabla u|_f \cdot \bm n_e) \delta_e[\bm n_e\bm n_e^{\mathrm T}]\in \bm V^1.
\end{equation}
Here we have used the facts that $u$ is piecewise linear (thus $\grad\grad u$ vanishes on each cell), $u$ is continuous (thus integrating by parts the first $\div$ does not lead to jump terms), and $\varphi$ vanishes on the boundary (thus no boundary terms appear). 

For $\sigma=\sum_{e\in \mathsf E_0}\sigma_e\delta_e[ \bm n_e\bm n_e^{\mathrm T}]\in \bm V^1$ and $\varphi\in C_c^{\infty}(\Omega; \mathbb R^{2})$, 
\begin{align*}
	&{-} \langle\rot \sigma,\varphi\rangle = \langle \sigma, \curl \varphi\rangle = \left\langle \sum_{e\in \mathsf E_0}\sigma_e \delta_e[\bm n_e\bm n_e^{\mathrm T}], \curl \varphi\right\rangle =\sum_{e\in \mathsf E_0}\sigma_e\int_e\bm n_e\bm n_e^{\mathrm T}:\curl\varphi \mathrm d s\\
	=& \sum_{e\in \mathsf E_0}\sigma_e\int_e\bm n_e\bm n_e^{\mathrm T}:\curl\varphi \mathrm d s=\sum_{e\in \mathsf E_0}\sigma_e\int_e\bm n_e\cdot\curl\varphi \bm n_e \mathrm d s\\
	=&\sum_{e\in \mathsf E_0}\sigma_e\bm n_e\cdot\int_e\grad\varphi \bm \tau_e \mathrm d s=\left\langle\sum_{v\in\mathsf V_0}\sum_{e:v\subset e}\sigma_e(\bm\tau_e\cdot\bm\tau_e^v)\delta_v[\bm n_e],\varphi\right\rangle,	\end{align*}
where $\bm \tau_{e}^v$ a unit vector parallel to $\bm\tau_e$ pointing to $v$ (thus $\bm\tau_e\cdot\bm\tau_e^v=\pm 1$ reflects the orientation).  Therefore,
	\[\rot \sigma={-} \sum_{v\in \mathsf V_0}\sum_{e:v\subset e} \mathcal O(v,e) \sigma_e\delta_v[\bm n_e]\in \bm V^2. \]
 This proves that the sequence \eqref{cplx:hessian-2d} is a complex. 
\end{proof}

In the rest of this section, we compute the cohomology of \eqref{cplx:hessian-2d}. 


The proof of \Cref{lem:closedness-hessian-2d} (part II) follows two steps. First, we construct an auxiliary sequence starting with piecewise polynomials without interelement continuity, see \eqref{cplx:hessian-L2-2d} below. A straightforward calculation shows that there is a correspondence between this sequence and the chain complex of $\Delta$. This thus identifies the cohomology of \eqref{cplx:hessian-L2-2d} with the homology  (with $\mathcal P_1$ coefficients) of the domain.
 Second, we form short exact sequences using \eqref{cplx:hessian-L2-2d} and \eqref{cplx:hessian-2d}. A diagram chase allows us to conclude with the cohomology of \eqref{cplx:hessian-2d}.

Next, we provide details of the above sketched proof. Consider an auxiliary sequence
\begin{equation}
\label{cplx:hessian-L2-2d}
\begin{tikzcd}
0 \ar[r] & V^0_{-} \ar[r,"\hess"] & \bm V^1_{-} \ar[r, "\rot"] & \bm V^2_{-} \ar[r] &  0.
\end{tikzcd}
\end{equation}

The space $V^0_-$ consists of piecewise linear, but discontinuous functions:  
\begin{equation}
    V^0_- = C^{-1}\mathcal{P}_1 := \{ u \in L^2(\Omega) : u|_{f} \in \mathcal P_1(f) \text{ for all } f \in \mathsf F\}.
 \end{equation}

      To proceed, we first consider  $\hess u$ for $u \in C^{-1}\mathcal{P}_1$ in the distributional sense. This resembles the procedure of \eqref{lem:closedness-hessian-2d}. Actually, for $\bm \sigma \in C_c^{\infty}(\Omega; \mathbb S)$, we have
      \begin{equation}
        \begin{split}
            \langle \hess u, \bm \sigma \rangle := & \langle u, \div\div \bm \sigma \rangle \\
            = & \sum_{f \in \mathsf F} \int_{f} u \cdot \div\div \bm \sigma \\ 
            = & \sum_{f \in \mathsf F} \int_{\partial f} u\cdot (\div\bm \sigma \bm n_{\partial f}) - \int_{f} \nabla u \cdot \div \bm \sigma  \\
            = & \sum_{e \in \mathsf E_0} \sum_{f \in \mathsf F} \mathcal O(e,f) \Big[ u|_f \cdot (\div \bm \sigma \cdot \bm n_{e}) - (\nabla u|_f) \cdot (\bm \sigma \bm n_{e})\Big]. 
        \end{split}
      \end{equation}

      This invokes us to define the following functionals: for each edge $e \in \mathsf E$ and $p \in \mathcal P_1$, we define the distribution $\hat{\bm v}_{e}^1[p]$ as 
      \begin{equation}
        \label{eq:hatv1_e-2d}
        \langle \hat{\bm v}_{e}^1[p], \bm \sigma \rangle = \int_{e} (\div \sym(\bm \sigma) \cdot \bm n_{e}) p  - \sym\bm \sigma : {(\nabla p\otimes {\bm n}_{e})}.
      \end{equation}

\begin{remark}
By our convention of notation, $\mathcal P_1$ is a bivariate function. Therefore, $\nabla p$ is well-defined. In evaluating the integrals, we restrict $p$ and $\nabla p$ to $e$, with a slight abuse of notation. {Recall that in standard notation of Discontinuous Galerkin method, $\sum_{f} \mathcal O(e,f) u|_f$ refers to the jump of $u$ on edge $e$, while $\sum_f \mathcal O(e,f) \nabla u|_f$ refers to the jump of $\nabla u$.}
\end{remark}

We define the space $\bm V^1_-$ to be the span of all the distributions $\hat{\bm v}_{e}^1[p]$, namely, 
      \begin{equation}
        \bm V^1_- = \Big\{\sum_{e \in \mathsf E_0} \hat{\bm v}_{e}^1[p_{e}] : p_{e} \in \mathcal P_1, \forall e \in \mathsf E_0 \Big\}.
      \end{equation}
 
\begin{lemma}
    \label{lem:dist-hess-L2-2d}
    For $u \in V^0_-$, it holds that
      $$\hess u = \sum_{e \in \mathsf E_0} \sum_{f: f\supset e} \mathcal O(e, f) \hat{\bm v}_e^1[u|_{f}].$$
    \end{lemma}

    
    \begin{remark}
      \label{rmk:ker-g1}
      We note that \Cref{lem:dist-hess-L2-2d} is consistent with \eqref{eq:dist-gradgrad-2d} if $u\in \bm V^{1}$, i.e., $u$ is continuous. 
        In fact, the continuity of $u$ implies $p = \sum_{f : f\supset e} \mathcal O(e,f) u|_f$ vanishes at $e$. Using the fact that $p\in \mathcal{P}_{1}$ and $p|_{ e} = 0$, we obtain $\nabla p = \alpha \bm n_{e}$ for some (scalar) constant $\alpha$. 
      Therefore,
        \begin{equation}
            \langle \hat{\bm v}_{e}^1[p],\bm \sigma \rangle = \alpha \int_{e} \bm \sigma :(\bm n_{e} \otimes \bm n_{e}) = \langle  \alpha \delta_{e}[\bm n_e \otimes \bm n_e], \bm \sigma \rangle  
        \end{equation}
This has the form of  \eqref{eq:dist-gradgrad-2d}.
    \end{remark}
    
      Finally, we consider $\rot$ of $\hat{\bm v}_{e}^1[p]$. Again, a direct calculation yields 
      \begin{equation}
        \begin{split}
        \langle \rot \hat{\bm v}_{e}^1[p], \bm v\rangle = & - \langle \hat{\bm v}_{e}^1[p], \curl \bm v\rangle \\ =
        & \int_{e} \curl \bm v : \sym( \nabla p\otimes\bm n_{e}) {-} (\div \sym \curl \bm v \cdot \bm n_{e})p.
        \end{split}
      \end{equation}
 It holds that 
    \begin{equation}
        \begin{split}
        \int_e \curl \bm v : \sym(\nabla p\otimes\bm n_{e} ) & = \int_e \sym \curl \bm v : (\nabla p\otimes\bm n_{e} ) \\
        & = \int_{e}  \curl \bm v: (\nabla p\otimes\bm n_{e} ) - \frac{1}{2} \begin{bmatrix} 0 & \div \bm v \\ {-}\div \bm v & 0 \end{bmatrix} : (\nabla p\otimes\bm n_{e} ) \\
        & =  \int_{e} \frac{\partial }{\partial \bm t_e} ( \nabla p\cdot\bm v) - \frac{1}{2} \int_{e} \div \bm v \cdot \frac{\partial   p}{\partial \bm t_e},
        \end{split}
    \end{equation}
    and 
    \begin{equation}
      \begin{split}
 \int_{e} (\div \sym \curl \bm v \cdot \bm n_{e})p & = \frac{1}{2} \int_e (\curl(\div \bm v) \cdot \bm n_e) p = \frac{1}{2} \int_{e} \frac{\partial}{\partial \bm t_e}(\div \bm v) p.
      \end{split}
    \end{equation}

Therefore, 
\begin{equation}
\begin{split}
\langle \rot \hat{\bm v}_{e}^1[p], \bm v \rangle & =  \int_{e} \frac{\partial}{\partial \bm t_e}(\bm v\cdot \nabla p - \frac{1}{2} \div \bm v \cdot p) =  (\bm v\cdot \nabla p - \frac{1}{2} \div \bm v \cdot p)|_{x_1}^{x_2}.
\end{split}
\end{equation}

We then define the final space $\bm V^2_-$ as follows:
\begin{equation}
\bm V^2_- = \Big\{ \sum_{x \in \mathsf V_0} \hat{\bm v}_{ x}^2[p_{x}] : p_{ x} \in \mathcal P_1 \Big\}, 
\end{equation}
where $\hat{\bm v}_{ x}^2[p_{ x}]$ is defined as 
\begin{equation}
  \label{eq:hatv2-2d}
\langle \hat{\bm v}_{x}^2[p_{x}], \bm w\rangle = (\bm w \cdot \nabla p - \frac{1}{2} \div \bm w \cdot p)(x).
\end{equation}
 
The above calculation yields the following lemma.
\begin{lemma}
  \label{lem:dist-rot-L2-2d}
For $\hat{\bm v}_{e}^1[p] \in \bm V^1_-$, 
\begin{equation}
\rot \hat{\bm v}_{e}^1[p] =  \sum_{x \in \mathsf V_0} \mathcal O(x, e) \hat{\bm v}_{x}^2[p].
\end{equation}
As a result, $\rot \bm V^1_-  \subset \bm V^2_-.$
\end{lemma}

Combining both \Cref{lem:dist-hess-L2-2d} and \Cref{lem:dist-rot-L2-2d}, we have
\begin{proposition}
  \label{prop:hom-hessian-L2-2d}
The sequence \eqref{cplx:hessian-L2-2d} is a complex, and its cohomology is isomorphic to $\mathcal{H}^{\bs}_{dR}(\Omega)\otimes \mathcal P_1 $. 
\end{proposition}
\begin{proof}
  The above calculation can be summarized in the following diagram.
  \begin{equation}\label{cd:mapping-kappa-2d}
    \begin{tikzcd}
   0 \arrow{r} & V^0_- \arrow{r}{\hess} \arrow{d}{\kappa^0_-} &\bm V^1_- \arrow{r}{\rot} \arrow{d}{\kappa^1_-} & \bm V^2_- \arrow{r}{ } \arrow{d}{\kappa^2_-}& 0\\
    0 \arrow{r}&C_2(\Delta, \mathcal P_1;\partial \Delta) \arrow{r}{\partial}&C_1(\Delta, \mathcal P_1;\partial \Delta)  \arrow{r}{\partial} & C_0(\Delta, \mathcal P_1;\partial \Delta) \arrow{r}{} & 0.
     \end{tikzcd}
    \end{equation}
Here $C_k(\Delta, \mathcal P_1;\partial \Delta)$ is the space of simplicial  $k$-chains with $\mathcal{P}_{1}$ coefficients, see \Cref{sec:topology}. 
For each element in $C_2(\Delta, \mathcal P_1; \partial \Delta)$, we can represent it as 
\begin{equation}
\omega = \sum_{f \in \mathsf F} p_f \|f\|,
\end{equation}
where $\|f\|$ is the free element associated with the face $f$. Similar expressions will be used later for 1-chain and 0-chain. 

The vertical mapping $\kappa^0_-$ maps $v \in V^0_-$ to 
$$\kappa^0_-(v) = \sum_{f \in \mathsf F} v|_{f} \|f\|.$$
We can similarly define linear maps $\kappa^1_-$ and $\kappa^2_-$ by evaluating the coefficients, i.e., 
$$
\kappa^1_-(\hat{\bm v}_{e}^1[p]):= p \|e\|, \quad  \kappa^2_-(\hat{\bm v}_{x}^2[q]):=q \|x\|.
$$
Obviously, $\kappa_{-}^{i}, =0, 1, 2$ are bijective and the diagram \eqref{cd:mapping-kappa-2d} commutes. 
 As a consequence, $\kappa_{-}^{\bs}$ induces bijective maps on cohomology, and therefore the cohomology $\mathcal H^{k}(V^{\bs}_{-})$ is isomorphic to $\mathcal H_{2-k}(\Delta, \mathcal P_1;\partial \Delta)$, which is further isomorphic to $\mathcal H_{2-k}(\Delta; \partial \Delta)\otimes \mathcal P_1\cong \mathcal H^{k}_{dR}(\Omega)\otimes \mathcal P_1$ by Theorems \ref{UCT} and \ref{thm:deRham}.
\end{proof}
\begin{remark}
In fact, the complex \eqref{cplx:hessian-L2-2d} can be regarded as the ``skeleton'' of the finite element divdiv complex \cite{hu2021family}, playing the role of Whitney forms for high-order de~Rham complexes. Similar structures have been used in \cite{hu2023local,hu2023local2}.

\end{remark}

Now we consider the cohomology of the original complex \eqref{cplx:hessian-2d}. The key is to regard \eqref{cplx:hessian-2d} as the kernel of \eqref{cplx:hessian-L2-2d} under certain maps. This leads to short exact sequences and the following diagram: 
\begin{equation}\label{cd:mapping-g-2d}
  \begin{tikzcd}
    0 \arrow{r} & 0 \arrow{r}{\hess} \arrow{d} &\bm V^1 \arrow{r}{\rot} \arrow{d} & \bm V^2 \arrow{r} \arrow{d}& 0 \\
 0 \arrow{r} & V^0_- \arrow{r}{\hess} \arrow{d}{g^0} &\bm V^1_- \arrow{r}{\rot} \arrow{d}{g^1} & \bm V^2_- \arrow{r} \arrow{d}{g^2}& 0\\
  0 \arrow{r}& \bigoplus\limits_{f \in \mathsf F} \mathcal P_1(f) \arrow{r}{\widetilde{\partial}}&\bigoplus\limits_{e \in \mathsf E_0} \mathcal P_1(e)  \arrow{r}{\widetilde{\partial}} &\bigoplus\limits_{x \in \mathsf V_0} \mathcal P_1(x)  \arrow{r}{} & 0.
   \end{tikzcd}
  \end{equation}

First, we introduce the last row of \eqref{cd:mapping-g-2d}: 
\begin{equation}
  \label{cplx:tildepartial-2d}
  \begin{tikzcd}
    0 \arrow{r}& \bigoplus\limits_{f \in \mathsf F} \mathcal P_1(f) \arrow{r}{\widetilde{\partial}}&\bigoplus\limits_{e \in \mathsf E_0} \mathcal P_1(e)  \arrow{r}{\widetilde{\partial}} &\bigoplus\limits_{x \in \mathsf V_0} \mathcal P_1(x)  \arrow{r}{} & 0.
  \end{tikzcd}
\end{equation}
Here  $\widetilde{\partial}$ is a composition of the boundary operator and the restriction operators, defined as 
\begin{equation}
\widetilde{\partial}( \sum_{f \in \mathsf F} p_{f} \|f\|) = \sum_{e \in \mathsf E_0}  \sum_{f \in \mathsf F}\mathcal O(e,f) (p_f)|_e \|e\|,
\end{equation}
\begin{equation}
  \widetilde{\partial}( \sum_{e \in \mathsf E_0} p_{e} \|e\|) = \sum_{x \in \mathsf V_0}  \sum_{e \in \mathsf E_0} \mathcal O(x,e) (p_e)|_x \|x\|.
  \end{equation}

The following results were proved in \cite[Lemma 4.9]{licht2017complexes}.
\begin{proposition}
\label{prop:exactness-tildepartial-2d}
The sequence \eqref{cplx:tildepartial-2d} is a complex, and its cohomology is $V^0$, $0$ and $0$, respectively.
\end{proposition}

  For completeness, we provide a proof here. The proof is based on the following lemma.

  \begin{lemma}
  The complex \eqref{cplx:tildepartial-2d} can be identified with a direct sum of the following vertex patch complexes for each vertex $v$,
\begin{equation}
\label{cplx:tildepartial-2d-vertexpatch}
\begin{tikzcd}
    0 \arrow{r}& 
    \bigoplus\limits_{f \in \mathsf F, f\ni v} \mathbb R \arrow{r}{{\partial}}&\bigoplus\limits_{e \in \mathsf E_0, e \ni v} \mathbb R  \arrow{r}{{\partial}} & \mathbb R  \arrow{r}{} & 0.
  \end{tikzcd}
\end{equation}
Here $\partial$ is induced by the boundary operator (the relative homology version).
\end{lemma}
\begin{proof}
The key observation is that for each simplex $f\in\mathsf F$, the space $\mathcal P_1(f)$ is generated by the barycenter coordinates $\lambda_v$, where $v$ is a vertex of $f$. By this, we can rewrite $\mathcal P_1(f) \cong \bigoplus_{v \in f}  \mathbb R$.
Therefore,
\begin{equation}
\bigoplus_{f\in\mathsf F} \mathcal P_1(f) \cong \bigoplus_{f \in \mathsf F} \bigoplus_{v \in f} \mathbb R \cong \bigoplus_{v \in \mathsf V} \bigoplus_{f \in \mathsf F, f\ni v} \mathbb R.
\end{equation}
For $f \in \mathsf F$ and $\lambda_v \in \mathcal P_1(f)$ with $v \in f$, by definition it holds that 
\begin{equation}\label{sum-v}
\widetilde{\partial} \lambda_v|_{f} = \sum_{e \in \mathsf E_0, f\ni e} \mathcal O(e, f) \lambda_v|_{e}.
\end{equation}
For $e$ such that $v \not\in e$,   $\lambda_v$ vanishes on $e$. {Therefore, the right-hand side of \eqref{sum-v} contains only terms that are in the patch of $v$. This proves that \eqref{cplx:tildepartial-2d} can be identified as a complex on each vertex patch}.
\end{proof}

\begin{proof}[Proof of \Cref{prop:exactness-tildepartial-2d}]
Note that the homology of \eqref{cplx:tildepartial-2d-vertexpatch} is identical to the relative homology of the vertex patch of $v$, where $v$ is either an interior or a boundary vertex. As a vertex patch is contractible, the homology vanishes except for at index zero, where $\mathcal H_0 \cong \operatorname{span}\{ \lambda_v\}$, the span of the Lagrange basis (hat) function at $v$.
\end{proof}

Now we are ready to finish the proof of \Cref{lem:closedness-hessian-2d}. 

\begin{proof}[Proof of \Cref{lem:closedness-hessian-2d} (part II)]
Recall the diagram \eqref{cd:mapping-g-2d}. 
We now define the vertical maps in \eqref{cd:mapping-g-2d} as the restriction to subcells, i.e., $g^{0}=I$, $g^1: \hat{\bm v}_{e}^1[p]\mapsto p|_{e}\|e\|$, and $g^2: \hat{\bm v}_{x}^2[p]\mapsto p|_{x} \|x\|$.
It is not difficult to see that each $g^k$ is onto. 
Now we prove that the space $\bm V^1=\ker(g^{1})$ and $\bm V^2=\ker(g^{2})$.  The former is implied by \Cref{rmk:ker-g1}. To see the latter, note that, by \eqref{eq:hatv2-2d}, $p|_{x} = 0$ implies  
$$\langle \hat{\bm v}_{x}^2[p_{x}], \bm w\rangle = (\bm w \cdot \nabla p)(x),$$
which is exactly all the vertex delta. Therefore, the kernel of $g^2$ is $\bm V^2$.

The above argument shows that columns of \eqref{cd:mapping-g-2d} are short exact sequences. Thus, it induces a long exact sequence of homologies:
\begin{equation}
  \begin{tikzcd}
    0 \arrow{r} & 0 \arrow{r} \arrow{d} &\ker(\rot:\bm V^1\to \bm V^2) \arrow{r}\arrow{d} & \mathcal H^2(V^{\bs}) \arrow{r} \arrow{d}& 0 \\
 0 \arrow{r} &   \mathcal H^0_{dR}(\Omega)\otimes\mathcal P_1 \arrow[r, ""{coordinate, name=Z}] \arrow{d} & \mathcal  H^1_{dR}(\Omega) \otimes\mathcal P_1  \arrow[r, ""{coordinate, name=Y}] \arrow{d} &  \mathcal H^2_{dR}(\Omega) \otimes\mathcal P_1 \arrow{r} \arrow{d}& 0\\
  0 \arrow{r}& V^0 \arrow{r} \arrow[uur, rounded corners, dashed, to path={ -- ([yshift=-4ex]\tikztostart.south)
  -| (Z) [near end]\tikztonodes
  |- ([yshift=4ex]\tikztotarget.north)
  -- (\tikztotarget)}]
  & 0 \arrow{r} \arrow[uur, rounded corners, dashed, to path={ -- ([yshift=-4ex]\tikztostart.south)
  -| (Y) [near end]\tikztonodes
  |- ([yshift=4ex]\tikztotarget.north)
  -- (\tikztotarget)}] &0  \arrow{r}{} & 0.
   \end{tikzcd}
  \end{equation}
With a diagram chase, we can verify that the connecting map $V^0 \to \ker(\rot : \bm V^1 \to \bm V^2)$ is Hessian.  

Now we compute the cohomology. For $\mathcal H^0(V^{\bs}) = \ker(\hess : V^0 \to \bm V^1)$, by the exactness of the long sequence, $\mathcal H^0(V^{\bs}) \cong \mathcal H^0_{dR}(\Omega)\otimes \mathcal P_1 $. For $\mathcal H^1(V^{\bs}) = {\ker(\rot: \bm V^1 \to \bm V^2)}/{\hess V^0}$, 
note that the following part from the long sequence
$$
\begin{tikzcd}
V^{0} \arrow{r}{\hess}& \ker(\rot: \bm V^1 \to \bm V^2)  \arrow{r}{j}& \mathcal H^1_{dR}(\Omega)\otimes\mathcal P_1 \to 0
\end{tikzcd}
$$
is exact. 
Therefore, ${\ker(\rot: \bm V^1 \to \bm V^2)}/{\hess V^0} \cong \ran(j) =   \mathcal H^1_{dR}(\Omega)\otimes\mathcal P_1.$ {Here we used the fact that $f: X\to Y$ induces an isomorphism $X/\ker(f) \cong \ran(f).$}

Finally, it holds that $\mathcal H^2(V^{\bs}) \cong \mathcal H^2_{dR}(\Omega)\otimes\mathcal P_1  $.
\end{proof}

\subsection{2D Hessian complex with homogeneous boundary condition}

In this section, we introduce the discrete Hessian complex with homogeneous boundary conditions (HBCs) and compute its cohomology.

\begin{lemma}
The sequence \eqref{cplx:hessian0-2d} is a complex.
\end{lemma}
\begin{proof}
We check the complex property by definition. 
For $u \in V^0_0$ and $\bm \sigma \in C^{\infty}_c(\mathbb R^2;\mathbb S)$, we have 
\begin{equation}
\begin{split}
 \langle \hess_0 u , \bm \sigma \rangle & = \langle  u, \div\div \bm \sigma \rangle \\
 & =  \sum_{f \in \mathsf F} \int_f u \div\div \bm \sigma \\ 
 & = - \sum_{f \in \mathsf F} \int_f\nabla u\cdot \div \bm \sigma  \\
 & = - \sum_{e \in \mathsf E} \sum_{f :f \supset e}  \int_e\Big[\mathcal O_0(e,f) (\nabla u|_f) \cdot \bm n_e\Big] \left(\bm n_e\cdot \bm \sigma\cdot \bm n_e\right). 
\end{split}
\end{equation}
{The last line uses the fact that the jump of $\nabla u$ only has the normal component.}
Note that $v|_{\partial \Omega} = 0$ implies that $\frac{\partial v}{\partial \bm t}|_{\partial \Omega} = 0.$
The remaining proof is similar to those in \Cref{lem:closedness-hessian-2d}.
\end{proof}



Now we consider the cohomology of \eqref{cplx:hessian0-2d}, and intuitively we can guess the result is $\mathcal H_{2-k}(\Delta)\otimes \mathcal P_1$, where $\mathcal  H_{\bs}(\Delta)$ is the simplicial homology. We verify this claim below.
\begin{theorem}
\label{thm:hom-hessian0-2d}
The cohomology of \eqref{cplx:hessian0-2d} is isomorphic to $ \mathcal H_{2 - k}(\Delta)\otimes\mathcal P_1$, therefore to $ \mathcal H^{k}_{dR,c}(\Omega)\otimes \mathcal P_1.$
\end{theorem}

The proof is analogous to the Hessian complex without homogeneous boundary conditions. 
We first introduce $V^0_{0,-}: = V^0_{-}$,
$$\bm V^1_{0,-} :=  \Big\{ \sum_{e \in \mathsf E}\hat{\bm v}^1_{e}[p_e]: p_e \in \mathcal P_1\Big\},
\text{  and  }
\bm V^2_{0,-} :=  \Big\{ \sum_{x \in \mathsf V}\hat{\bm v}^2_{x}[p_x]: p_x \in \mathcal P_1\Big\}.$$

Using the same technique, we can prove that 
\begin{proposition}
\label{prop:hom-hessian0-L2-2d}
The sequence
\begin{equation}
\begin{tikzcd}
0 \ar[r] &  V^0_{0,-} \ar[r, "\hess_0"] &  \bm V^1_{0,-} \ar[r, "\rot_0"]&  \bm  V^2_{0,-} \ar[r] & 0
\end{tikzcd}
\end{equation}
is a complex, and the cohomology is $  \mathcal H^{\bs}_{dR,c}(\Omega)\otimes\mathcal P_1.$
\end{proposition}
The detailed proof of \Cref{prop:hom-hessian0-L2-2d} can be found in Appendix.

To complete, we now introduce the boundary version of \eqref{cplx:tildepartial-2d}:
\begin{equation}
  \label{cplx:tildepartial0-2d}
  \begin{tikzcd}
    0 \arrow{r}& \bigoplus_{f \in \mathsf F} \mathcal P_1(f) \arrow{r}{\widetilde{\partial}_0}&\bigoplus_{e \in \mathsf E} \mathcal P_1(e)  \arrow{r}{\widetilde{\partial}_0} &\bigoplus_{x \in \mathsf V} \mathcal P_1(x)  \arrow{r}{} & 0.
  \end{tikzcd}
\end{equation}

Here  $\widetilde{\partial}$ is defined as 
\begin{equation}
\widetilde{\partial}_0( \sum_{F \in \mathsf F} p_{f} \|f\|) = \sum_{f \in \mathsf F} \sum_{e \in \mathsf E} \mathcal O_0(e,f) (p_f)|_e \|e\|,
\end{equation}
\begin{equation}
  \widetilde{\partial}_0( \sum_{e \in \mathsf E} p_{e} \|e\|) = \sum_{x \in \mathsf V} \sum_{e \in \mathsf E} \mathcal O_0(x,e) (p_e)|_x \|x\|.
  \end{equation}

\begin{proposition}
\label{prop:exactness-tildepartial0-2d}
The sequence \eqref{cplx:tildepartial-2d} is  a complex, and its cohomology is $V_0^0$, $0$ and $0$, respectively.
\end{proposition}
\begin{proof}
The proof is similar to that of \Cref{prop:exactness-tildepartial-2d}, and can be found in Appendix. 
\end{proof}

With these two results, we can prove \Cref{thm:hom-hessian0-2d} by investigating the diagram and its cohomology below:
\begin{equation}\label{cd:mapping-g-2d-0}
  \begin{tikzcd}
    0 \arrow{r} & 0 \arrow{r}{\hess_0} \arrow{d} &\bm V^1_0 \arrow{r}{\rot_0} \arrow{d} & \bm V^2_0 \arrow{r} \arrow{d}& 0 \\
 0 \arrow{r} & V^0_{0,-} \arrow{r}{\hess_0} \arrow{d}{g^0} &\bm V^1_{0,-} \arrow{r}{\rot_0} \arrow{d}{g^1} & \bm V^2_{0,-} \arrow{r} \arrow{d}{g^2}& 0\\
  0 \arrow{r}& \bigoplus\limits_{f \in \mathsf F} \mathcal P_1(f) \arrow{r}{\widetilde{\partial}_0}&\bigoplus\limits_{e \in \mathsf E} \mathcal P_1(e)  \arrow{r}{\widetilde{\partial}_0} &\bigoplus\limits_{x \in \mathsf V} \mathcal P_1(x)  \arrow{r}{} & 0.
   \end{tikzcd}
  \end{equation}
  Here $g^k$'s are the same as those in the proof of \Cref{lem:closedness-hessian-2d}. The only difference comes from the homology of the lower row. The desired result follows from a chase on the following diagram:
  \begin{equation}
    \begin{tikzcd}
      0 \arrow{r} & 0 \arrow{r} \arrow{d} &\ker(\rot_0:\bm V^1_0\to \bm V^2_0) \arrow{r}\arrow{d} & \mathcal H^2(V^{\bs}_0) \arrow{r} \arrow{d}& 0 \\
   0 \arrow{r} &   \mathcal H^0_{dR,c}(\Omega)\otimes \mathcal P_1\arrow[r, ""{coordinate, name=Z}] \arrow{d} &  \mathcal  H^1_{dR,c}(\Omega) \otimes\mathcal P_1 \arrow[r, ""{coordinate, name=Y}] \arrow{d} &  \mathcal H^2_{dR,c}(\Omega)\otimes\mathcal P_1  \arrow{r} \arrow{d}& 0\\
    0 \arrow{r}& V^0_0 \arrow{r} \arrow[uur, rounded corners, dashed, to path={ -- ([yshift=-4ex]\tikztostart.south)
    -| (Z) [near end]\tikztonodes
    |- ([yshift=4ex]\tikztotarget.north)
    -- (\tikztotarget)}]
    & 0 \arrow{r} \arrow[uur, rounded corners, dashed, to path={ -- ([yshift=-4ex]\tikztostart.south)
    -| (Y) [near end]\tikztonodes
    |- ([yshift=4ex]\tikztotarget.north)
    -- (\tikztotarget)}] &0  \arrow{r}{} & 0.
     \end{tikzcd}
    \end{equation}

\subsection{2D divdiv complex}

In this section, we focus on the divdiv complex in 2D, showing that the sequence \eqref{cplx:divdiv-2d} is a complex and proving the cohomology is isomorphic to $ \mathcal H^{\bs}_{dR}(\Omega)\otimes \mathcal{RT}$ (\Cref{thm:hom-divdiv-2d}, part I).

\begin{proposition}
The sequence \eqref{cplx:divdiv-2d} is  a complex.
\end{proposition}
\begin{proof}
It is straightforward to see that $\sym\curl$ maps $\bm U^0$ to $\bm U^1$. For the $\div\div$ part, it suffices to note that for $u \in C_c^{\infty}(\Omega)$, 
\begin{equation}
    \label{eq:dist-divdiv-2d}
    \begin{split}
\langle \div\div \bm \sigma, u \rangle = & \langle \bm \sigma, \nabla^2 u \rangle \\
= & \sum_{ f \in \mathsf F}\int_{ f} \bm \sigma : \nabla^2 u \\
=& \sum_{e \in \mathsf E_0} \sum_{f : f \supset e} \int_{e}\Big[ \mathcal O(e,f) [\bm \sigma|_f \bm n_e] \Big] \cdot \nabla u \\
= & \sum_{e \in \mathsf E_0}  \sum_{f : f \supset e} \int_e\Big[ \mathcal O(e,f)  [ \bm t_e \cdot \bm \sigma|_f \cdot  \bm n_e] \Big]  \frac{\partial u}{\partial \bm t_e} \\
= &  \sum_{e \in \mathsf E_0}  \sum_{f : f \supset e}  \Big[ \mathcal O(e,f) [ \bm t_e \cdot \bm \sigma|_f \cdot  \bm n_e] \Big]  (u(x_2) - u(x_1)),
    \end{split}
\end{equation} 
where $x_1$ and $x_2$ are the two vertices of $e$. 
The third line is due to the normal-normal continuity of $\bm \sigma$.  
This implies that $\div\div \bm U^1 \subset U^2$.
\end{proof}

It is nontrivial to compute the cohomology of \eqref{cplx:divdiv-2d} directly. Nevertheless, we can regard the divdiv complex \eqref{cplx:divdiv-2d} as the dual of the Hessian complex with homogeneous boundary conditions \eqref{cplx:hessian0-2d}. This allows us to conclude the cohomology of \eqref{cplx:divdiv-2d} from the results for the Hessian complex that we have obtained. The duality is explained in the following diagram and will be explained below:
\begin{equation}\label{cd:dual-pairing-V0-U}
    \begin{tikzcd}[column sep = large]
   0 \arrow{r} & V^0_0 \arrow{r}{\hess_0} \arrow[d, phantom, "\ast"] &\bm V^1_0 \arrow{r}{\rot_0} \arrow[d, phantom, "\ast"] & \bm V^2_0 \arrow{r}{} \arrow[d, phantom, "\ast"]& 0\\
    0 &U^2 \arrow{l}& \bm U^1 \arrow{l}[swap]{\div\div} & \bm U^0 \arrow{l}[swap]{\sym\curl}  & 0\arrow{l}{},
     \end{tikzcd}
\end{equation}
where $\ast$ denotes duality.
Observe that in each of the vertical lines in \eqref{cd:dual-pairing-V0-U}, the degrees of freedom and the delta functions change their roles:
\begin{itemize}
    \item[-] $V^0_0$ is the scalar Lagrange finite element, whose degrees of freedom are at each vertex. Correspondingly, the distribution space $U^2$ consists of the vertex deltas. The relationship of $\bm V^2_0$ and $\bm U^0$ shares the same idea, but is in a vector setting.
    \item[-] For the tensor case,  the distribution space $\bm V^1_0$ is spanned by the normal-normal deltas, while $\bm U^1$ is the piecewise constant function with normal-normal continuity.
\end{itemize}

This invokes us using dual pairing to show that the two horizontal complexes of \eqref{cd:dual-pairing-V0-U} have the same homology. 
    With slight abuse of notation, we still use $\langle \cdot , \cdot\rangle$ to denote the family of dual pairing, but we do not require that the functions are smooth. For example, for $\hat{v} = \sum_{x \in \mathsf V_0} a_{x} \delta_{x}$, we define the dual pairing
    $
    \langle v, \hat{v} \rangle = \sum_{x \in \mathsf V_0} a_{x} v(x).
    $ The expression is well-defined whenever $v$ is single-valued at each internal vertex. Restricting the dual pairing to the spaces $\bm U^2 \times V^0_0$ will give a non-degenerate one. Similarly, we can define the dual pairing between $\bm U^0$ and $\bm V^2_0$. For the tensor case, for $\hat{\bm \sigma} = \sum_{e \in \mathsf E} a_{e} \delta_{e}[\bm n_{e} \otimes \bm n_{e}]$, we define 
    \begin{equation}
    \langle \bm \sigma, \hat{\bm \sigma}\rangle = \sum_{e \in \mathsf E} a_{e} \int_e (\bm n_{e} \cdot \bm \sigma \cdot\bm n_{e})|_{e}.
    \end{equation}
    The dual pairing is well-defined since $\bm n_e \cdot \bm \sigma \cdot \bm n_e$ is single-valued.


The next proposition shows that with the pairs specified above, the two complexes are adjoint to each other.


\begin{proposition}
The divdiv complex \eqref{cplx:divdiv-2d} is adjoint to the Hessian complex {with boundary conditions} in \eqref{cplx:hessian0-2d}, namely, 
\begin{equation}
    \label{eq:dual-2d-id1}
    \langle v, \div\div \bm \sigma\rangle_{V^0_0 \times \bm U^2} = \langle \hess_0 v, \bm \sigma \rangle_{\bm V^1_0 \times \bm U^1}, \quad v \in V^0_0, \bm \sigma \in \bm U^1,
\end{equation}
and 
\begin{equation}
    \label{eq:dual-2d-id2}
    \langle \bm \sigma, \sym\curl \bm u\rangle _{V^1_0 \times \bm U^1} = \langle {-}\rot_0 \sigma, \bm u\rangle_{V^2_0 \times \bm U^0}, \quad \bm \sigma \in \bm V^1_0, \bm u \in \bm U^0.
\end{equation}
\end{proposition}
\begin{remark}
These identities resemble the definition of distributional derivatives. The difference is that the differential operators on both sides here are in the sense of distributions. As a result, the identities hold valid only on the given finite/distributional element pair, which relies on the fact that the functions are piecewise constant or linear.
\end{remark}
\begin{proof}


We first show \eqref{eq:dual-2d-id1}. 
For $u\in V^0_0$ and $\bm \sigma\in  \bm U^1$, note that $\int_f \div\div(u \bm \sigma) = 0$ on each face $f$ as $u$  is linear and $\sigma$ is constant (here $\div\div$ is defined on $f$ for smooth functions in the usual sense). This implies 
{
\begin{equation}
\begin{split}
0 & = \sum_{e \in \mathsf E} \sum_{f: f\supset \mathsf F} \int_e\Big[ \mathcal O_0(e,f) \div(u|_f \bm \sigma) \Big]\cdot \bm n_e \\
& = \sum_{e \in \mathsf E} \int_e \Big[ \mathcal O_0(e,f) \frac{\partial}{\partial \bm n_e} (u|_f \bm \sigma\bm n_e)\cdot \bm n_e \Big] + \Big[ \mathcal O_0(e,f) \frac{\partial}{\partial \bm t_e} (u\bm \sigma\bm t_e)\cdot \bm n_e  \Big] \\ 
& = \sum_{e \in \mathsf E}\sum_{f: f\supset e} \int_e \Big[ \mathcal O_0(e,f) \frac{\partial}{\partial \bm n_e} u|_f \Big] (\bm n_e \cdot \bm \sigma \cdot \bm n_e) + \int_e \Big[ \mathcal O_0(e,f) \bm t_e \cdot \bm \sigma|_f \cdot \bm n_e \Big] \frac{\partial}{\partial \bm t_e} u.
\end{split}
\end{equation}}
Here the last line is due to the fact that $u$ and $\bm n_e \cdot \bm \sigma \cdot \bm n_e$ are continuous across the edge. It then follows from \eqref{eq:dist-divdiv-2d} (together with the fact that $u$ vanishes on $\mathsf E_{\partial}$), and \eqref{eq:dist-gradgrad-2d} that 
$$ 0 = \langle v, \div\div \bm \sigma \rangle_{V_0^0 \times \bm U^2} - \langle \hess_0 v, \bm \sigma \rangle_{\bm V_0^1 \times \bm U^1}.$$
This proves \eqref{eq:dual-2d-id1}.

Next, we have for $\bm u \in \bm U^0$, and $e \in \mathsf E$,
\begin{equation}
\begin{split}
\langle \sym\curl \bm u, \delta_{e}[\bm n_{e} \otimes \bm n_{e}] \rangle = & \int_{e} \bm n_{e} \cdot \curl \bm u \cdot \bm n_{e} \\ 
= & \int_{e} \bm n_{e} \cdot \nabla \bm u \cdot \bm t_{e} \\ 
= & \bm u \cdot \bm  n_{e}(x_2) - \bm u \cdot \bm n_{e}(x_1) \\ 
= & \langle \bm u, {-}\rot \delta_{e}[\bm n_{e} \otimes \bm n_{e}]\rangle,
\end{split}
\end{equation}
where $e=[x_1,x_2]$. 
This proves the second identity.
\end{proof}



\begin{proof}[The proof of \Cref{thm:hom-divdiv-2d} (part I)]

We first define $\pi_{\bm V^0_0 \to U^2}$ such that $\pi_{\bm V^0_0 \to U^2}(v) = \sum_{x\in \mathsf V} v(x)\delta_x.$ Similarly, we define $\pi_{\bm V^1_0 \to \bm U^1}$, $\pi_{\bm V^2_0 \to \bm U^0}$. Define $\pi_{\bm U^k \to \bm 
V^{2-k}_0}, k = 0,1,2,$ to be their inverses. 

Next, we define suitable inner products on the spaces $U^{k}$ and $V^{2-k}_0$.
For $V^0_0$, we define the inner product 
\begin{equation}
(u, u')_{V^0} = \sum_{x \in \mathsf V_0} u(x) u'(x).
\end{equation}
For $U^2$, we define the inner product 
\begin{equation}
(\sum_{x \in \mathsf V_0} u_x \delta_x, \sum_{x \in \mathsf V_0} u'_x\delta_x)_{U^2 } = \sum_{x \in \mathsf V_0} u_xu'_x.
\end{equation}

Therefore, we have 
\begin{equation}
(u, u')_{V^0_0} = \langle u, \pi_{V^0_0 \to \bm U^2}u' \rangle_{V^0_0 \times \bm U^2} = (\pi_{V^0_0 \to \bm U^2}u, \pi_{V^0_0 \to \bm U^2}u')_{U^2}.
\end{equation}
Similarly, we can define the inner product and the one-to-one mappings for the remaining spaces. 
Now we can regard the complexes \eqref{cd:dual-pairing-V0-U} as two Hilbert complexes (with finite dimensional spaces) to compute the (co)homology.


Our goal is to show that $\mathcal H^2(U^{\bs}) \cong \mathcal H^{2-k}(V_0^{\bs})$. 
The key is that when specifying an inner product, the cohomology can be represented by the harmonic forms. For example, the cohomology $\mathcal H^1(V^{\bs}_0)$ can be represented via $\bm \sigma \in \bm V_0^1$, satisfying
$
\rot_0 \bm \sigma = 0
$
and 
$
(\bm \sigma, \hess u)_{V^1_0} = 0.
$
Such a $\bm \sigma$ is called a harmonic form of $\bm V^1_0$. 

Now we show that $\bm \sigma$ is a harmonic form of $\bm V^1_0$ if and only if $\pi_{\bm V^1_0 \to \bm U^1} \bm \sigma$ is a {harmonic form} of $\bm U^1$. This follows from direct calculations.
\begin{equation}
\begin{split}
& \rot_0 \bm \sigma  = 0 \\
\Longleftrightarrow& \langle \rot_0 \bm \sigma, \bm \xi\rangle_{\bm V^1_0 \times \bm U^1} = 0,\quad \forall \bm \xi \in \bm U^1\\
\Longleftrightarrow& \langle \bm \sigma, \sym\curl \bm \xi\rangle_{\bm V^1_0 \times \bm U^1} = 0,\quad \forall \bm \xi \in \bm U^1\\
 \Longleftrightarrow &(\pi_{\bm V^1_0 \to \bm U^1}, \bm \sigma, \sym \curl \bm U^1)_{\bm U^1} = 0. 
\end{split}
\end{equation}
And 
\begin{equation}
\begin{split}
& (\bm \sigma, \hess_0 u) = 0 ,\quad \forall u \in V^0_0\\ 
 \Longleftrightarrow & 
 \langle \pi_{\bm V^1_0 \to \bm U^1} \bm \sigma, \hess_0 u \rangle_{\bm U^1 \times \bm V^1_0} = 0, \quad \forall u \in V^0_0
 \\ 
 \Longleftrightarrow & 
 \langle \div\div \pi_{\bm V^1_0 \to \bm U^1} \bm \sigma, u \rangle_{U^2\times V^0_0} = 0, \quad \forall u \in V^0_0
 \\  \Longleftrightarrow &  \div\div \pi_{\bm V^1_0 \to \bm U^1} \bm \sigma = 0. 
\end{split}
\end{equation}
This indicates that $\pi_{\bm V^1_0 \to \bm U^1}$ induces an isomorphism. We can show that the other $\pi_\bs$ operators are isomorphisms similarly.

It follows that $\mathcal H^{k}(U^{\bs}) \cong \mathcal H^{2-k}(V^{\bs}_0) \cong  \mathcal H^{2-k}_{dR,c}(\Omega)\otimes\mathcal P_1 $. The theorem is proved by noticing that $\mathcal H^{2-k}_{dR,c}(\Omega) \cong \mathcal H^{k}_{dR}(\Omega).$ 
\end{proof}

 \subsection{2D divdiv complex with homogeneous boundary condition}

We can also consider a divdiv complex with homogeneous boundary condition:
\begin{equation}
    \label{cplx:divdiv0-2d}
    \begin{tikzcd}[column sep = large]
0\ar[r] & \bm U^0_0 \ar[r,"\sym\curl_0"] &  \bm U^1_0 \ar[r,"\div\div_0"]  & U^2_0 \ar[r] &  0,
    \end{tikzcd}
\end{equation}
where $\bm U^0_0$ is the (vector) piecewise Lagrange element with zero boundary conditions, 
\begin{equation}
    \bm U^1_0 := \{ \bm \sigma \in \bm U^1: (\bm n_{e} \cdot \bm \sigma \cdot \bm n_{e})|_{\partial\Omega} = 0\},
    \end{equation}
    and
    \begin{equation}
    U^2_0 := \operatorname{span}\{ \delta_{x}: x \in \mathsf V\}.
    \end{equation}

Regarding \eqref{cplx:divdiv0-2d} as the dual complex of \eqref{cplx:hessian-2d}, we can prove \Cref{thm:hom-divdiv-2d} (part II), which is re-summarized in the following theorem. The detailed proof can be found in Appendix.
\begin{theorem}
\label{thm:hom-divdiv0-2d}
The sequence \eqref{cplx:divdiv0-2d} is a complex, with cohomology isomorphic to $ \mathcal H^{\bs}_{dR,c}(\Omega)\otimes \mathcal{RT}$. 
\end{theorem}

\section{Complexes in 3D}
\label{sec:3d}

\subsection{3D Hessian complex}

Note that $\delta_{f}[\bm A]$ is linear with respect to the tensor $\bm A$, therefore, we can rewrite 
$c_{f}\delta_{f}[\bm n_{f} \otimes \bm n_{f}] = \delta_{f }[c_{f} \bm n_{f} \otimes \bm n_{f}]$ for any real number $c_{f} \in \mathbb R$. Moreover, the element $\bm \sigma \in \bm V^1$ can be expressed as 
\begin{equation}
    \label{eq:V1repr-3d}
\bm \sigma = \sum_{f \in \mathsf F_0} \delta_{f}[\bm m_{f} \otimes \bm n_{f}],
\end{equation}
where $\bm m_{f} \in \mathbb R^3$ is a normal vector to the face $f$ (not necessarily unit).

\begin{proposition}
The sequence \eqref{cplx:hessian-3d} is a complex.
\end{proposition}
\begin{proof}
Let $u \in V^0$. For any  $\bm \varphi \in C_{c}^{\infty}(\Omega; \mathbb R^{3\times 3})$, it holds that
\begin{equation}
\begin{split}
    \langle \hess u, \bm \varphi \rangle & = \langle u, \div \div \bm \varphi \rangle = \int_{\Omega} u \div \div \bm \varphi  \\ 
    & = - \int_{\Omega} \grad u \cdot \div \bm \varphi \\ 
    & = - \sum_{f\in \mathsf F_0} \sum_{K : K \supset f} \int_{f} \Big[ \mathcal O(f,K) \nabla u|_K  \Big] \cdot (\bm \varphi\cdot \bm n_{f})  \\
    &= -\sum_{f\in \mathsf F_0} \sum_{K: K \supset f} \int_{f} \Big[ \mathcal O(f,K)  \frac{\partial u|_K}{\partial \bm n} \Big] \cdot (\bm n_{f} \otimes \bm n_{f}: \bm \varphi).
\end{split}
\end{equation} 
Here, the second line uses the fact that $u$ is continuous, and the third line uses the fact that the piecewise gradient of $\grad u$ is zero. The last line is due to that $u$ is continuous, and thus, the jump of $u$ across the face has vanishing tangential components.
Therefore, we have 
\begin{equation}
\hess u = -\sum_{f \in \mathsf F_0} \sum_{K : K \supset f} \Big[ \mathcal O(f,K) \frac{\partial u|_K}{\partial \bm n} \Big]\, \delta_{f}[\bm n_{f} \otimes \bm n_{f}].
\end{equation}


    Let $\bm \sigma \in \bm V^1$, which has the expression
    \begin{equation}
        \bm \sigma = \sum_{f \in \mathsf F_0} \delta_{f}[\bm m_{f} \otimes \bm n_{f}],
    \end{equation}
    where $\bm m_{f}$ is normal to face $f$, $f \in \mathsf F_0$.
    Then it holds that
    \begin{equation}
        \label{eq:dist-curl-3d}
        \curl \bm \sigma = \sum_{e \in \mathsf E_0} \sum_{f: f \supset  e} \mathcal O(e, f) \delta_{e}[\bm m_{f} \otimes \bm t_{e}].
    \end{equation}
  Consequently, we have $\curl \bm V^1 \subseteq \bm V^2$.

We now show \eqref{eq:dist-curl-3d}. By linearity, it suffices to show that the identity \eqref{eq:dist-curl-3d} holds for a single term $\delta_{f} [\bm m_{f} \otimes \bm n_{f}]$. We follow the definition: for $\bm \sigma \in C^{\infty}(\mathbb R^3; \mathbb R^{3\times 3})$, it holds that,
\begin{equation}
\begin{split}
    \langle \curl \delta_{f}[\bm m_{f} \otimes \bm n_{f}], \bm \sigma\rangle  = &\langle \delta_{f}[\bm m_{f} \otimes \bm n_{f}], \curl \bm \sigma\rangle 
    = 
    \int_{f} \bm m_f \cdot \curl \bm \sigma \cdot \bm n_f \\
    = & \int_{\partial f} \bm m_{f} \cdot \bm \sigma \cdot \bm t_{\partial f}.
\end{split}
\end{equation}
Taking $\bm \sigma \in C_c^{\infty}(\Omega; \mathbb R^{3\times 3})$ completes the proof of \eqref{eq:dist-curl-3d}.

    Let $\bm \tau \in \bm V^2$, and denote 
    $$ \bm \tau = \sum_{e \in \mathsf E_0} \delta_{e}[\bm m_{e} \otimes \bm t_{e}],$$
    where $\bm m_{e}$ is normal to the edge $e$.
    Then, it holds that 
    \begin{equation}\label{eq:dist-div-3d}
        \div \bm \tau = -\sum_{x \in \mathsf V_0} \sum_{e: e \ni x} \mathcal O(x, e)\delta_{x}[\bm m_{e}] .
    \end{equation}
    This implies that, $\div \bm V^2 \subseteq \bm V^3.$

It suffices to check the case for $\div \delta_{e}[\bm m_{e} \otimes \bm t_{e}]$. For $\bm u \in C^{\infty}_c(\Omega)$, it holds that
\begin{equation}
    \begin{split}
\langle \div \delta_{e}[\bm m_{e} \otimes \bm t_{e}], \bm u \rangle = & \langle \delta_{e}[\bm m_{e} \otimes \bm t_{e}], - \nabla \bm u\rangle \\
= & - \int_{e} \bm m_{e} \cdot \frac{\partial \bm u}{\partial \bm t_{e}} \\ 
= & (\bm m_{e} \cdot \bm u)(x_1) - (\bm m_{e} \cdot \bm u)(x_2).
    \end{split}
\end{equation}
Here $e = [x_1, x_2]$. This completes the proof.

\end{proof}

Here we use the following lemma, which is a direct consequence of the Stokes' formula.
\begin{lemma}
For an oriented face $f$ and a vector-valued function $\bm u$, we have 
\begin{equation}
\int_{f} \curl \bm u \cdot \bm n = \int_{\partial f} \bm u \cdot \bm t,
\end{equation}
where the direction of the unit normal vector $\bm n$ of $f$, and that of the unit tangential vector $\bm t$ of $\partial f$, are determined by the given orientation.
\end{lemma}

Analogous to the two-dimensional case, we introduce the following auxiliary Hessian complex starting with a discrete $L^2$ space:
\begin{equation}
    \label{cplx:hessian-L2-3d}
    \begin{tikzcd}
   0  \ar[r] &   V^0_- \ar[r,"\hess"] &  \bm V^1_- \ar[r,"\curl"] & \bm V^2_-  \ar[r,"\div"] & \bm V^3_- \ar[r] & 0.
    \end{tikzcd}
  \end{equation}
Here, the superscript $-$ denotes that this complex starts with a discontinuous element.

\subsubsection*{Construction of $V^0_-$} The space is chosen as the discontinuous piecewise linear function space, i.e., 
\begin{equation}
   V^0_- = C^{-1}\mathcal P_1 := \{ u \in L^2(\Omega) : u|_{K} \in \mathcal P_1(K) \text{ for all cells } K \in \mathsf K\}.
\end{equation}
{We can identify the space to $\bigoplus_{K \in \mathsf K} \mathcal P_1 = C_3(\Delta, \mathcal P_1; \partial \Delta)$ by $\kappa^0_-(u) := \sum_{K \in \mathsf K} u|_K \|K\|. $}
\subsubsection*{Construction of $\bm V^1_-$} 

For each face $f \in\mathsf F$ and $p \in \mathcal P_1$, we first define the following distribution:

\begin{equation}\label{eq:hatv1-3d}
    \langle \hat{\bm v}^1_{f}[p], \bm \sigma \rangle = \int_{f} (\div \sym(\bm \sigma) \cdot \bm n_f) p - \int_{f} \sym \bm \sigma:   (\nabla p\otimes\bm n_f) ,
\end{equation}
for $\bm \sigma \in C^{\infty}_c(\Omega; \mathbb R^{3\times 3})$.


We now define the space $\bm V^1_-$ to be the span of all the distribution $\hat{\bm v}_{f}^1[p]$, namely,
\begin{equation}
\bm V^1_- = \Big\{ \sum_{f \in \mathsf F_0} \hat{\bm v}^1_{f}[p_{f}] : f \in \mathsf F_0, p_{f} \in \mathcal P_1 \Big\}.
\end{equation}


The following proposition states that $\hess V^0_- \subseteq \bm V^1_-$.

\begin{proposition}
\label{prop:dist-hess-L2-3d}
For $u \in V^0_-$, we have 
\begin{equation}\label{eq:dist-hess-L2-3d}
    \hess u =
    \sum_{f \in \mathsf F_0}
     \sum_{K : K \supset f}  \mathcal O(f, K) \hat{\bm v}^1_{f}[u|_{K}] \end{equation} 
in the sense of distributions.
Consequently, we have $\hess V^0_- \subseteq \bm V^1_-.$
\end{proposition}
\begin{proof}
By definition, for symmetric-matrix valued function $\bm \sigma \in C_c^{\infty}(\Omega; \mathbb S)$ we have 
\begin{equation}
\begin{split}
\langle \hess u, \bm \sigma \rangle = & \langle u, \div\div \bm \sigma \rangle =  \int_{\Omega} u \div\div \bm \sigma  \\ 
= & - \int_{\Omega } \nabla u \cdot \div \bm\sigma + \sum_{f \in \mathsf F_0}\sum_{K : K \supset f}   \int_{f}\Big[ \mathcal O(f,K) u|_{K} \Big] \div \bm \sigma \cdot \bm n_{f}  \\ 
= & -\sum_{f \in \mathsf F_0}\sum_{K : K \supset f} \int_f \Big[\mathcal O(f,K) \nabla u|_K \Big]\cdot (\bm \sigma \bm n_{f}) + \sum_{f \in \mathsf F_0} \sum_{K : K \supset f} \int_f \Big[\mathcal O(f,K) u|_K \Big]\div \bm \sigma \cdot \bm n_{f} \\
= & - \sum_{f \in \mathsf F_0}  \sum_{K : K \supset f}\int_f \bm \sigma : ( \Big[\mathcal O(f,K) \nabla u|_K \Big] \otimes\bm n_{f} ) + \sum_{f \in \mathsf F_0} \sum_{K : K \supset f} \int_f \Big[\mathcal O(f,K) u|_K \Big] \div \bm \sigma \cdot \bm n_{f}.
\end{split}
\end{equation}
Since $\hess u$ is symmetric, it then holds that 
\begin{equation}
\begin{split}
    \langle \hess u, \bm \sigma \rangle = & \langle \hess u, \sym \bm \sigma \rangle \\ 
    = & - \sum_{f \in \mathsf F_0} \sum_{K: K\supset f}  \int_f \bm \sigma : \sym ( \Big[ \mathcal O(f,K) \nabla u|_K \Big]\otimes\bm n_{f} )\\
    & \quad\quad\quad+ \sum_{f \in \mathsf F_0}\sum_{K: K\supset f}  \int_f \Big[ \mathcal O(f,K) \nabla u|_K \Big] \div \sym \bm \sigma \cdot \bm n_{f}.
\end{split}
\end{equation}
This completes the proof of \eqref{eq:dist-hess-L2-3d}, and therefore $\hess V^0_- \subseteq \bm V^1_-.$
\end{proof}

We define the following mapping: 
$$\kappa^1_- : \bm V^1_- \to \bigoplus_{f \in \mathsf F_0} \mathcal P_1 = C_2(\Delta, \mathcal P_1; \partial \Delta), $$
by
\begin{equation}
\kappa^1_-(\sum_{f \in \mathsf F_0} \hat{\bm v}_{f}^1[p_{f}]) = \sum_{f \in \mathsf F_0} p_{f} \| f\|.
\end{equation}
Here $\|f\|$ is the free element associated to the face $f$.

\subsubsection*{Construction of $\bm V^2_-$} The space $\bm V^2_-$ contains the following traceless distribution $\hat{\bm v}_{e}^2[p]$ for each internal edge $e \in \mathsf E_0$ and $p \in \mathcal P_1$:

\begin{equation}
\langle \hat{\bm v}_{e}^2[p], \bm \sigma \rangle  = -\int_{e} \dev \bm \sigma : ( \nabla p \otimes\bm t_{e}) + \frac{1}{2}\int_{e} (\div \dev \bm \sigma \cdot \bm t_{e} ) p,
\end{equation}
for any test function $\bm \sigma \in C^{\infty}_c(\mathbb R^3;\mathbb R^{3\times 3}).$

By definition, $\hat{\bm v}_{e}^2[p]$ is a traceless distribution. 
\begin{proposition}
\label{prop:dist-curl-L2-3d}
For $\bm \sigma \in \bm V^1_-$,  denote that $\bm \sigma= \sum_{f \in \mathsf F_0} \hat{\bm v}^1_{f}[p_{f}]$. Then we have 
\begin{equation}
 \curl \bm \sigma  = \sum_{e \in \mathsf E_0} \sum_{f : f \supset e} \mathcal O(e, f) \hat{\bm v}_{e}^2[p_{f}].
\end{equation}
This implies
$\curl \bm \sigma \in \bm V^2_-$
in the sense of distributions.
\end{proposition}
\begin{proof}
It suffices to check a single term $\hat{\bm v}^1_{f}[p] \in \bm V^1_-$. By definition, for a matrix-valued function $\bm \sigma \in C^{\infty}_c(\Omega; \mathbb T)$, it follows  that
\begin{equation}
\begin{split}
\langle \curl \hat{\bm v}^1_{f}[p], \bm \sigma \rangle =& \langle \hat{\bm v}^1_{f}[p], \sym \curl \bm \sigma \rangle \\
= & \int_{f} - (\sym \curl \bm \sigma \cdot \bm n_f) \cdot \nabla p + (\div (\sym \curl \bm \sigma) \cdot \bm n_f) p \\
= & \int_{f} -  (\curl \bm \sigma \cdot \bm n_f) \cdot \nabla p + \frac{1}{2} \int_{f} (\div \bm \sigma \times \bm n_{f}) \cdot \nabla p + \frac{1}{2} \int_{f} (\curl \div \bm \sigma \cdot \bm n_{f} ) p  \\ 
= & -\int_{\partial f} ( \bm \sigma\cdot\bm t_{\partial f} ) \cdot \nabla p + \frac{1}{2} \int_{f} \curl(\div \bm \sigma p) \cdot \bm n_f \\
= & -\int_{\partial f} (\bm \sigma \cdot \bm t_{\partial f}) \cdot \nabla p + \frac{1}{2} \int_{\partial f} (\div \bm \sigma \cdot \bm t_{\partial f}) p.
\end{split}
\end{equation}
Here the third line comes from the identity
$ (\sym \curl \bm \sigma)\bm n_f = (\curl \bm \sigma) \bm n_f - \frac{1}{2} \div \bm \sigma \times \bm n_f$, and the fourth line is due to the identity $\curl(p \bm w) = \nabla p \times \bm w + p \curl \bm w$.
Consequently, we have $\curl \hat{\bm v}^1_{f}[p] \in \bm V^2_-$.
\end{proof}

For $\bm \tau = \sum_{e \in \mathsf E_0} \hat{\bm v}^2_{e}[p_{e}] \in \bm V^2_-$, we define $\kappa^2_-: \bigoplus_{e \in \mathsf E_0} \mathcal P_1(e) = C_1(\Delta,\mathcal P_1; \partial \Delta)$ as 
\begin{equation}
\kappa^2_-(\bm \tau) = \sum_{e \in \mathsf E_0} p_{e} \|e\|.
\end{equation}

\subsubsection*{Construction of $\bm V^3_-$} Finally, let us consider the construction of $\bm V^3_-$. The space $\bm V^3_-$ consists of the following distribution $\hat{\bm v}^3_{x}[p]$ for each internal vertex $x \in \mathsf V_0$ and $p \in \mathcal P_1$:
\begin{equation}
\langle \hat{\bm v}^3_{x}[p], \bm q\rangle = (\bm q \cdot \nabla p)(x) - \frac{1}{3}(\div \bm q)p(x),
\end{equation}
for all $\bm q \in C_c^{\infty}(\mathbb R^3; \mathbb R^3)$. 
\begin{proposition}
\label{prop:dist-div-L2-3d}
For $\bm \tau = \sum_{e} \hat{\bm v}^2_{e}[p_{e}] \in \bm V^2_-$, we have 
$$\div \bm \tau = \sum_{x \in \mathsf V_0} \sum_{e: e \ni x} \mathcal O(x, e) \hat{\bm v}^3_{ x}[p_{e}]$$
in the distributional sense.
\end{proposition}
\begin{proof}
By definition, we have 
\begin{equation}
\begin{split}
\langle \div \hat{\bm v}^2_{e}[p_{e}], \bm q \rangle & = - \langle \hat{\bm v}^2_{e}[p_{e}], \dev\nabla \bm q \rangle \\ 
& = \int_{e} (\dev \nabla \bm q \cdot \bm t_{e}) \cdot \nabla p - \frac{1}{2}\int_{e} (\nabla \div \dev \bm q \cdot \bm t_{e})p 
\\ 
& = \int_{e} (\nabla \bm q \cdot \bm t_{e}) \cdot \nabla p - \frac{1}{3} \int_{e} \div q (\bm t_{e} \cdot \nabla p) - \frac{1}{3}\int_{e} (\nabla \div \bm q \cdot \bm t_{e})p 
\\ 
& = (\bm q \cdot \nabla p)(x_2) - \frac{1}{3} (p \div \bm q )(x_2) -  (\bm q \cdot \nabla p)(x_1) + \frac{1}{3} (p \div \bm q)(x_1).
\end{split}
\end{equation}
Here we use the fact that $\dev \nabla \bm q = \nabla \bm q - \frac{1}{3} \div \bm q I$ and $\div \dev \nabla \bm q = \frac{2}{3} \div \bm q.$
Here $e = [x_1, x_2]$.
\end{proof}
For $\bm \xi = \sum_{x \in \mathsf V_0} \hat{\bm v}_{x}^3[p_{x}] \in \bm V^3_-$, define 
$\displaystyle
\kappa^3_-(\bm \xi) = \sum_{x \in \mathsf V_0} p_{x}\|x\|.
$




We have the following result about the cohomology of the auxiliary complex.
\begin{theorem}
The sequence \eqref{cplx:hessian-L2-3d} is a complex. The cohomology
is isomorphic to $\mathcal H^{\bs}_{dR}(\Omega)\otimes \mathcal P_1$. 
\end{theorem}
\begin{proof}
By \Cref{prop:dist-hess-L2-3d}, \Cref{prop:dist-curl-L2-3d} and \Cref{prop:dist-div-L2-3d}, the following diagram commutes. 

{
\begin{equation}\label{cd:mapping-kappa-3d}
    \begin{tikzcd}[column sep = small]
   0 \arrow{r} & V^0_- \arrow{r}{\hess} \arrow{d}{\kappa^0_-} &\bm V^1_- \arrow{r}{\curl} \arrow{d}{\kappa^1_-} & \bm V^2_- \arrow{r}{\div} \arrow{d}{\kappa^2_-} &\bm V^3_- \arrow{r}{} \arrow{d}{\kappa^3_-} & 0\\
    0 \arrow{r}&C_3(\Delta, \mathcal P_1; \partial\Delta) \arrow{r}{\partial}&C_2(\Delta, \mathcal P_1; \partial\Delta)\arrow{r}{\partial}&C_1(\Delta, \mathcal P_1; \partial\Delta)\ar[r,"\partial"] & C_0(\Delta, \mathcal P_1; \partial\Delta) \arrow{r}{} & 0.
     \end{tikzcd}
    \end{equation}
}

Clearly, $\kappa^{\bs}_-$ are (cochain) isomorphisms. Consequently, we have 
\begin{equation}
\mathcal H^k(V^{\bs}_-) \cong \mathcal H_{3-k}(C{\bs}(\Delta, \mathcal P_1; \partial \Delta)).
\end{equation}

It then follows from the universal coefficient theorem  that $\mathcal H_{3-k}(C{\bs}(\Delta, \mathcal P_1; \partial \Delta)) \cong  \mathcal H_{3-k}(\Delta; \partial \Delta)\otimes\mathcal P_1  \cong  \mathcal H_{dR}^k(\Omega)\otimes \mathcal P_1.$
\end{proof}

Now we consider the cohomology of \eqref{cplx:hessian-3d}. In 3D, consider the following diagram: 
\begin{equation}\label{cd:mapping-g-3d}
  \begin{tikzcd}
    0 \arrow{r} & 0 \arrow{r}{\hess} \arrow{d} &\bm V^1 \arrow{r}{\curl} \arrow{d} & \bm V^2 \arrow{r}{\div} \ar[d]& \bm V^3 \arrow{d} \arrow{r}& 0 \\
 0 \arrow{r} & V^0_- \arrow{r}{\hess} \arrow{d}{g^0} &\bm V^1_- \arrow{r}{\curl} \arrow{d}{g^1} & \bm V^2_-  \arrow{r}{\div}\arrow{d}{g^2} & \bm V^3 \arrow{r} \arrow{d}{g^3} & 0\\
  0 \arrow{r}& \bigoplus\limits_{K \in \mathsf K} \mathcal P_1 \ar[r, "\widetilde{\partial}"] & \bigoplus\limits_{f \in \mathsf F_0} \mathcal P_1(f) \arrow{r}{\widetilde{\partial}}&\bigoplus\limits_{e \in \mathsf E_0} \mathcal P_1(e)  \arrow{r}{\widetilde{\partial}} &\bigoplus\limits_{x \in \mathsf V_0} \mathcal P_1(x)  \arrow{r}{} & 0.
   \end{tikzcd}
  \end{equation}

  Here the lower row is defined similarly to \eqref{cplx:tildepartial-2d}. We have the following result.
  \begin{proposition}
    \label{prop:exactness-tildepartial-3d}
    The sequence 
    $$
    \begin{tikzcd}[column sep = small]  0 \arrow{r}& \bigoplus_{K \in \mathsf K} \mathcal P_1 \ar[r, "\widetilde{\partial}"] & \bigoplus_{f \in \mathsf F_0} \mathcal P_1(f) \arrow{r}{\widetilde{\partial}}&\bigoplus_{e \in \mathsf E_0} \mathcal P_1(e)  \arrow{r}{\widetilde{\partial}} &\bigoplus_{x \in \mathsf V_0} \mathcal P_1(x)  \arrow{r}{} & 0
    \end{tikzcd}
    $$
    is a complex, and its cohomology is $V^0$, $0$ and $0$, respectively.
    \end{proposition}
    The proof of \Cref{prop:exactness-tildepartial-3d} can be found in Appendix.
Following the proof of \Cref{lem:closedness-hessian-2d}, we can prove \Cref{thm:hom-hessian-3d} using the above theorem.
 The cohomology can be calculated from the following long exact sequence.  
\begin{equation}
    \begin{tikzcd}[column sep=small]
      0 \arrow{r} & 0 \arrow{r} \arrow{d} &\ker(\curl:\bm V^1\to \bm V^2) \arrow{r}\arrow{d} & \mathcal H^2(V^{\bs}) \arrow{r} \arrow{d} & \mathcal H^3(V^{\bs}) \arrow{r} \arrow{d} & 0 \\
   0 \arrow{r} &  \mathcal H^0_{dR}(\Omega)\otimes \mathcal P_1 \arrow[r, ""{coordinate, name=Z}] \arrow{d} &  \mathcal  H^1_{dR}(\Omega)\otimes \mathcal P_1 \arrow[r, ""{coordinate, name=Y}] \arrow{d} &  \mathcal H^2_{dR}(\Omega) \otimes\mathcal P_1 \arrow[r, ""{coordinate, name=X}] \arrow{d}&  \mathcal H^3_{dR}(\Omega) \otimes \mathcal P_1\ar[r] \ar[d]&  0\\
    0 \arrow{r}& V^0 \arrow{r} \arrow[uur, rounded corners, dashed, to path={ -- ([yshift=-4ex]\tikztostart.south)
    -| (Z) [near end]\tikztonodes
    |- ([yshift=4ex]\tikztotarget.north)
    -- (\tikztotarget)}]
    & 0 \arrow{r} \arrow[uur, rounded corners, dashed, to path={ -- ([yshift=-4ex]\tikztostart.south)
    -| (Y) [near end]\tikztonodes
    |- ([yshift=4ex]\tikztotarget.north)
    -- (\tikztotarget)}] &0  \arrow{r} \arrow[uur, rounded corners, dashed, to path={ -- ([yshift=-4ex]\tikztostart.south)
    -| (X) [near end]\tikztonodes
    |- ([yshift=4ex]\tikztotarget.north)
    -- (\tikztotarget)}] & 0 \ar[r] & 0.
     \end{tikzcd}
    \end{equation}

\subsection{3D Hessian complex with homogeneous boundary conditions}

{This subsection focuses on the cohomology of \eqref{cplx:hessian0-3d}. We first introduce an auxiliary complex, which corresponds to the HBC version of \eqref{cplx:hessian-L2-3d}.}
We define $V^0_{0,-} = V^0_{-}$, $$\bm V^1_{0,-} = \Span\Big\{\hat{\bm v}^1_f[p] : f\in \mathsf F, p \in \mathcal P_1\Big\},$$
 $$\bm V^2_{0,-} = \Span\Big\{\hat{\bm v}^2_e[p]: e \in \mathsf E, p \in \mathcal P_1\Big\},$$
 and 
 $$\bm V^3_{0,-} = \Span\Big\{\hat{\bm v}^3_x[p]: x \in \mathsf V, p \in \mathcal P_1\Big\}.$$

 Then we have 
 \begin{theorem}
\label{thm:hom-hessian0-L2-3d}
The sequence 
\begin{equation}
\begin{tikzcd}
0 \ar[r] & V^0_{0,-} \ar[r,"\hess_0"] & \bm V^1_{0,-} \ar[r,"\curl_0"] & \bm V^2_{0,-} \ar[r,"\div_0"] & \bm V^3_{0,-} \ar[r] & 0
\end{tikzcd}
\end{equation}
is a complex, with cohomology isomorphic to $  \mathcal H^{k}_{dR,c}(\Omega)\otimes\mathcal P_1$.
 \end{theorem}

Similarly to the previous section, we can prove \Cref{thm:hom-hessian0-3d}, see Appendix for more details.



\subsection{3D divdiv complex}

We first prove \Cref{thm:unisolvency-U1}, the unisolvency of the $\mathbb T \oplus \bm x \times \mathbb S$ element. 

\begin{proof}[Proof of \Cref{thm:unisolvency-U1}]
    The number of the proposed degrees of freedom is equal to the number of shape functions. Therefore, to show the unisolvency, we only show that if all the degrees of freedom \ref{dof:symcurl1} and \ref{dof:symcurl2} vanish on a function $\bm u\in \mathbb T \oplus \bm x\times \mathbb S$, then $\bm u$ vanishes.

    Let $\bm u=\bm v+\bm x\times \bm \sigma$. We have for any face $f$,
    \begin{equation}\label{sigma-nn}
    -2\int_f \bm n_f\cdot\bm  \sigma\cdot\bm  n_f=\int_f \bm n_f\cdot \curl \bm u\cdot\bm n_f=\int_{\partial f}  \bm n_f \cdot  \bm u \cdot \bm n_f \times \bm \nu_{\partial f, f}=\int_{\partial f}\bm n_f\cdot \bm u\cdot \bm t_{\partial f}=0.
     \end{equation}
    Now \eqref{sigma-nn} and the fact that $\int_e \bm n_e\cdot \bm u\cdot \bm t_e =0$ imply that $\curl \bm u \in \mathbb B_{K}$. By the second set of degrees of freedom \ref{dof:symcurl2}, $\curl \bm u = 0$, which means that $\bm u$ is a constant.  Furthermore, we can conclude that $\bm u = 0$.
    
\end{proof}

{We set the bubble function space $\mathbb B_{K}^*$ to be the span of the dual basis of \ref{dof:symcurl2}. 
}

In Section \ref{sec:main-divdiv3d}, we introduced the divdiv complex \eqref{cplx:divdiv-3d}, which consists of a distributional $\dev\grad$ operator, a piecewise $\sym\curl$ operator and a discrete $\div\div$ operator. The first several spaces are finite elements with polynomial shape functions and locally defined degrees of freedom. To obtain a neat discrete topological interpretation, we eliminated the interior degrees of freedom and obtained \eqref{cplx:divdiv-}. To show the cohomology of these complexes, in this section, we view \eqref{cplx:divdiv-} from a different point of view as the dual of the distributional Hessian complex \eqref{cplx:hessian0-3d}.
We will use the duality structure and the cohomology of the Hessian complex (Theorem \ref{thm:hom-hessian0-3d}) to derive the cohomology of \eqref{cplx:divdiv-3d-trimmed}. 
This will prove  Theorem \ref{thm:hom-divdiv-3d}.


We observe that the spaces in \eqref{cplx:divdiv-} are dual to the spaces in the Hessian complex \eqref{cplx:hessian0-3d}. This inspires us to consider the following sequence with dual operators:
\begin{equation}
    \label{cplx:divdiv-3d-trimmed}
    \begin{tikzcd}[column sep = large]
    0 \arrow{r}& \bm U^0 \arrow{r}{\widetilde{\dev\grad}}& \widehat{\bm U}^1 \arrow{r}{\widetilde{\sym\curl}} & \widehat{\bm U}^2 \arrow{r}{\widetilde{\div\div}} & U^3\arrow{r}& 0.
    \end{tikzcd}
    \end{equation}

Here, the operators are defined as the dual of $\div_0$, $\curl_0$, and $\hess_0$ (in the sense of distributions) in the Hessian complex \eqref{cplx:hessian0-3d}. More specifically,
\begin{itemize}
\item[-] For $\bm u \in \bm U^0$, define $\widetilde{\dev\grad} \bm u \in \bm U^1$ such that $$\langle \widetilde{\dev\grad} \bm u, \bm \tau \rangle_{\bm U^1 \times \bm V^2_0} := {-} \langle \bm u, \div_0 \bm \tau \rangle_{\bm U^0 \times \bm V^3_0}, \quad \forall \bm \tau \in \bm V^2_0.$$ 
\item[-] For $\bm \xi \in \bm U^1$, define $\widetilde{\sym\curl} \bm \xi \in \bm U^2$ such that 
$$\langle \widetilde{\sym\curl} \bm \xi, \bm \sigma \rangle_{\bm U^2 \times \bm V^1_0} := \langle \bm \xi, \curl_0 \bm \sigma \rangle_{\bm U^1 \times \bm V^2_0},\quad \forall \bm \sigma \in \bm V^1_0.$$ 
\item[-] For $\bm \eta \in \bm U^2$, define $\widetilde{\div\div} \bm \eta \in \bm U^3$ such that 
$$\langle \widetilde{\div\div} \bm \eta, u \rangle_{\bm U^3 \times \bm V^0_0} := \langle \bm \eta, \hess_0 u \rangle_{\bm U^2 \times \bm V^1_0},\quad \forall u \in \bm V^0_0.$$ 
\end{itemize}
Since  these dual pairs are non-degenerate, the dual operators $\widetilde{\dev\grad}$, $\widetilde{\sym\curl}$, $\widetilde{\div\div}$ are well-defined by the right-hand side.
Similar to the proof of \Cref{thm:hom-divdiv-2d}, we can introduce an inner product on each space and prove the following result, see Appendix for more details.
\begin{theorem}
    \label{hom:divdiv-3d-trimmed}
    The sequence \eqref{cplx:divdiv-3d-trimmed}, with the spaces and operators are defined above, is a complex. The cohomology is isomorphic to $\mathcal H^{\bs}_{dR}(\Omega)\otimes \mathcal{RT}.$
\end{theorem}

In fact, we can verify that the operators in \eqref{cplx:divdiv-3d-trimmed} coincide with the operators in \eqref{cplx:divdiv-3d}, i.e., $\widetilde{\dev\grad} = \dev\grad$, $\widetilde{\sym\curl} = \sym\curl_h$, and $\widehat{\div\div}= \widetilde{\div\div}$. Thus, the two complexes are identical. The details can be found in the Appendix. {As a consequence,  the sequence \eqref{cplx:divdiv-} is a complex with the same cohomology.}

Note that the divdiv complex \eqref{cplx:divdiv-3d} is a direct sum of \eqref{cplx:divdiv-} and 
\begin{equation}
\begin{tikzcd}[column sep=large]
0 \ar[r] & 0 \ar[r] & \mathbb B_K^* \ar[r,"\sym\curl_h"] & \mathbb B_K \ar[r] & 0 \ar[r] & 0.
\end{tikzcd}
\end{equation}
This concludes with the cohomology of \eqref{cplx:divdiv-3d} (\Cref{thm:hom-divdiv-3d}).


For the divdiv complex with homogeneous boundary conditions \eqref{cplx:divdiv0-3d}, the construction and the proof of the cohomology (\Cref{thm:hom-divdiv0-3d}) are similar. The key is to eliminate the interior degrees of freedom and identify the resulting complex as the dual of the Hessian complex \eqref{cplx:hessian-3d}, see Appendix.

\section{Conclusions and outlook}\label{sec:conclusion}

In this paper, we constructed some distributional BGG complexes. The canonical degrees of freedom of the resulting spaces allow a discrete topological and geometric interpretation. Although we focus on 2D and 3D, we hope the discussions will shed light on discretizing high dimensional and high order tensors on triangulation with potential applications in broad areas such as numerical geometric PDEs and exterior calculus on graphs and in graphics \cite{lim2020hodge,Wang:2023:ECIG}. This paper also developed a strategy for showing the cohomology of distributional complexes by constructing auxiliary sequences and using diagram chase. Consequently, we demonstrated the cohomology of the Regge finite element/sequence in 2D, which was open to the best of our knowledge. A similar idea may also be used to compute the cohomology of the 3D Regge sequence. However, further investigation is beyond the scope of this paper. 

Unless Whitney forms for the de~Rham complexes, the resulting spaces for the BGG complexes are less regular. 
  Solving PDEs with these schemes calls for further numerical analysis. 

Another example of the BGG construction is the conformal (deformation) complex \cite{arnold2021complexes,vcap2022bgg}. Although there has been progress on conforming finite element discretizations \cite{conformal2023}, distributional versions encoding discrete conformal geometric structures remain open. 

We investigated the connections between distributional finite elements and Discrete Exterior Calculus. This shift of point of view might provide another perspective for establishing convergence analysis for DEC and other lattice methods.

\appendix
\section{Technical Proofs}

In the Appendix, we include some technical proofs. 

\subsection{Operators in the 3D divdiv complex}

We verify that the operators in \eqref{cplx:divdiv-3d-trimmed} defined by duality are identical to those in \eqref{cplx:divdiv-}.
\begin{lemma}
In $\bm U^0$, we have $\widetilde{\dev\grad} = \dev\grad$. 
\end{lemma}
\begin{proof}
    Since $\bm u$ is $\bm H^1$ conforming, the $\dev\grad$ operator is piecewise. It follows that $\dev \grad \bm u \in \dev \grad \mathcal P_1 = \mathbb T$, which is in the local shape function space of $\widehat{\bm U}^1$. 
    It follows from \eqref{eq:div-formula} that $ \langle {\dev\grad} \bm u, \delta_{e}[ \bm n_{e,\pm}\otimes \bm t_{e}] \rangle_{\widehat{\bm U}^1 \times \bm V^2_0} = \langle \bm u, -\div \delta_{e}[ \bm n_{e,\pm}\otimes \bm t_{e}]\rangle_{\bm U^0 \times \bm V^3_0}.$
    This implies that, $\widetilde{\dev\grad} = \dev\grad$ on $\bm U^0$.
    
\end{proof}

\begin{lemma}
In $\bm U^1$, we have $\widetilde{\sym\curl} = \sym\curl_h$, the piecewise $\sym\curl$ operator.
\end{lemma}

\begin{proof}
It follows from \eqref{eq:curlV1}.
\end{proof}

\begin{lemma}
For any function $\bm \sigma \in \bm U^2$, it holds that $\widehat{\div\div} \bm \sigma = \widetilde{\div\div} \bm \sigma.$
\end{lemma}
\begin{proof}
It follows from 
\begin{equation}
\begin{split}
\langle \bm \sigma , \hess  u \rangle_{\bm U^2 \otimes \bm V^1_0} = & {-} \langle \bm \sigma, \sum_{f \in \mathsf F_0} \sum_{K: K \supset f} \Big[ \mathcal O(f,K) \frac{\partial u|_K}{\partial \bm n} \Big] \delta_f[{\bm n_f \otimes \bm n_f}] \rangle_{\bm U^2 \otimes \bm V^1_0} \\ 
= & {-}\sum_{f \in \mathsf F_0} \sum_{K: K \supset f} \Big[ \mathcal O(f,K) \frac{\partial u|_K}{\partial \bm n} \Big]  \int_f (\bm n_f \cdot \bm \sigma \cdot \bm n_f).
\end{split}
\end{equation}

Subtracting the right-hand side of \eqref{eq:hat-divdiv-3d} from the above equation leads to 
\begin{equation}
\langle \div \div \sigma, \bm u \rangle-\langle \bm \sigma , \hess  u \rangle_{\bm U^2 \otimes \bm V^1_0}=\sum_{f \in \mathsf F_0} \sum_{K : K \supset f } \int_f \Big[\mathcal O(f,K) \nabla u|_K \cdot \bm \sigma \cdot \bm n_f \Big] = 0.
\end{equation}
Here the first identity comes from a tangential-normal decomposition of $\nabla u$ and $\bm \sigma$, and the second is due to $\div(\bm \sigma \cdot \nabla u) = 0$ in each element $K$, as $u$ is linear and $\bm \sigma$ is constant. 
\end{proof}

\subsection{Diagrams illustrating the proof}

In the main body of this paper, some theorems follow from a similar argument. For example, to prove the cohomology of the Hessian complex \eqref{cplx:hessian-2d}, we first consider an auxiliary Hessian complex \eqref{cplx:hessian-L2-3d}, and relate it to a simplicial homology. Then, we use diagram chasing to obtain the cohomology of the original Hessian complex \eqref{cplx:hessian-2d}. For the divdiv complex \eqref{cplx:divdiv-2d}, we use a duality argument, and using isomorphisms between harmonic forms to derive the cohomology. For brevity, we omitted some details in the main text. In this appendix, we provide diagrams and a sketch of the omitted proofs.


\begin{proof}[Proof of \Cref{prop:hom-hessian0-L2-2d}]
The proof is based on the following diagram:
    \begin{equation}
        \begin{tikzcd}
       0 \arrow{r} & V^0_{0,-} \arrow{r}{\hess} \arrow{d}{\kappa^0_-} &\bm V^1_{0,-} \arrow{r}{\rot} \arrow{d}{\kappa^1_-} & \bm V^2_{0,-} \arrow{r}{ } \arrow{d}{\kappa^2_-}& 0\\
        0 \arrow{r}&C_2(\Delta, \mathcal P_1) \arrow{r}{\partial_0}&C_1(\Delta, \mathcal P_1)  \arrow{r}{\partial_0} & C_0(\Delta, \mathcal P_1) \arrow{r}{} & 0.
         \end{tikzcd}
        \end{equation}
The vertical maps induce an isomorphism between $\mathcal H(V^{k}_{0,-})$ and $\mathcal H_{2-k}(\Delta, \mathcal P_1)$, which is isomorphic to $\mathcal H^{\bs}_{dR,c}(\Omega) \otimes \mathcal P_1$.
\end{proof}

\begin{proof}[Proof of \Cref{prop:exactness-tildepartial0-2d}]
The homology can be identified with a direct sum of 
    \begin{equation}
        \begin{tikzcd}
            0 \arrow{r}& 
            \bigoplus\limits_{f \in \mathsf F, f\ni v} \mathbb R \arrow{r}{{\partial}}&\bigoplus\limits_{e \in \mathsf E, e \ni v} \mathbb R  \arrow{r}{{\partial}} & \mathbb R  \arrow{r}{} & 0.
          \end{tikzcd}
        \end{equation}
        For an interior vertex $v$, the homology is the relative homology (with respect to boundary) of the local patch of $v$, which vanishes except for at index zero. For a boundary vertex $v$, the homology is the relative homology (with respect to boundaries that are not the boundary of $\Delta$) of the local patch of $v$, which vanishes for all indices. Therefore, we conclude the result.
\end{proof}

\begin{proof}[Proof of \Cref{prop:exactness-tildepartial-3d}]
Using the fact that the original complex can be identified with the direct sum of 
    \begin{equation}
        \begin{tikzcd}
            0 \arrow{r}& 
            \bigoplus\limits_{K \in \mathsf K, K\ni v} \mathbb R \ar[r] & 
            \bigoplus\limits_{f \in \mathsf F, f\ni v} \mathbb R \arrow{r}{{\partial}}&\bigoplus\limits_{e \in \mathsf E, e \ni v} \mathbb R  \arrow{r}{{\partial}} & \mathbb R  \arrow{r}{} & 0.
          \end{tikzcd}
        \end{equation}
\end{proof}

\begin{proof}[Proof of \Cref{thm:hom-divdiv0-2d}]
We introduce the following dual pairs,
    \begin{equation}
        \begin{tikzcd}[column sep = large]
       0 \arrow{r} & V^0 \arrow{r}{\hess} \arrow[d, phantom, "\ast"] &\bm V^1 \arrow{r}{\rot} \arrow[d, phantom, "\ast"] & \bm V^2 \arrow{r}{} \arrow[d, phantom, "\ast"]& 0\\
        0 &U^2_0 \arrow{l}& \bm U^1_0 \arrow{l}[swap]{\div\div_0} & \bm U^0_0 \arrow{l}[swap]{\sym\curl_0}  & 0\arrow{l}{},
         \end{tikzcd}
    \end{equation}
and the inner products. For example, the dual pair of $v \in V^0$ and $\sum_{x \in \mathsf v} a_x \delta_x \in U^2_0$ is $\langle v, \sum_{x \in \mathsf V} a_x \delta_x \rangle = \sum_{x \in \mathsf V} a_x v(x)$, and the inner products are $(v,v')_{V^0} = \sum_{x \in \mathsf V} v(x)v'(x)$, and $(\sum_{x \in \mathsf V} a_x \delta_x, \sum_{x \in \mathsf V} a_x' \delta_x) = \sum_{x \in \mathsf V} a_x a_x'.$ Moreover, we introduce $\pi_{V^0 \to U^2_0}(v) := \sum_{x\in \mathsf V} v(x)\delta_x$ and vice versa. 

Similarly, we define other linear maps. It can be shown that $\pi_{V^k \to U^{2-k}_0}$ induces isomorphisms $\mathcal H^{k}(U^{\bs}_0) \cong \mathcal H^{2-k}(V^{\bs})$, which concludes the result.
\end{proof}

\begin{proof}[Proof of \Cref{thm:hom-hessian0-L2-3d}]
By the following diagram:
\begin{equation}
    \begin{tikzcd}[column sep = small]
   0 \arrow{r} & V^0_{-,0} \arrow{r}{\hess} \arrow{d}{\kappa^0_-} &\bm V^1_{-,0} \arrow{r}{\curl} \arrow{d}{\kappa^1_-} & \bm V^2_{-,0} \arrow{r}{\div} \arrow{d}{\kappa^2_-} &\bm V^3_{-,0} \arrow{r}{} \arrow{d}{\kappa^3_-} & 0\\
    0 \arrow{r}&C_3(\Delta, \mathcal P_1) \arrow{r}{\partial_0}&C_2(\Delta, \mathcal P_1)\arrow{r}{\partial_0}&C_1(\Delta, \mathcal P_1)\ar[r,"\partial_0"] & C_0(\Delta, \mathcal P_1) \arrow{r}{} & 0.
     \end{tikzcd}
    \end{equation}

\end{proof}

\begin{proof}[Proof of \Cref{thm:hom-hessian0-3d}]
The following diagram 
\begin{equation}
    \begin{tikzcd}
      0 \arrow{r} & 0 \arrow{r}{\hess_0} \arrow{d} &\bm V^1_0 \arrow{r}{\curl_0} \arrow{d} & \bm V^2_0 \arrow{r}{\div_0} \ar[d]& \bm V^3_0 \arrow{d} \arrow{r}& 0 \\
   0 \arrow{r} & V^0_{-,0} \arrow{r}{\hess_0} \arrow{d}{g^0} &\bm V^1_{-,0} \arrow{r}{\curl_0} \arrow{d}{g^1} & \bm V^2_{-,0}  \arrow{r}{\div_0}\arrow{d}{g^2} & \bm V^3 \arrow{r} \arrow{d}{g^3} & 0\\
    0 \arrow{r}& \bigoplus\limits_{K \in \mathsf K} \mathcal P_1 \ar[r, "\widetilde{\partial}_0"] & \bigoplus\limits_{f \in \mathsf F} \mathcal P_1(f) \arrow{r}{\widetilde{\partial}_0}&\bigoplus\limits_{e \in \mathsf E} \mathcal P_1(e)  \arrow{r}{\widetilde{\partial}_0} &\bigoplus\limits_{x \in \mathsf V} \mathcal P_1(x)  \arrow{r}{} & 0,
     \end{tikzcd}
    \end{equation}
    induces the long exact sequence below:
    \begin{equation}
        \begin{tikzcd}[column sep=small]
          0 \arrow{r} & 0 \arrow{r} \arrow{d} &\ker(\curl_0:\bm V^1_0\to \bm V^2_0) \arrow{r}\arrow{d} & \mathcal H^2(V^{\bs}_0) \arrow{r} \arrow{d} & \mathcal H^3(V^{\bs}_0) \arrow{r} \arrow{d} & 0 \\
       0 \arrow{r} &  \mathcal H^0_{dR,c}(\Omega)\otimes \mathcal P_1 \arrow[r, ""{coordinate, name=Z}] \arrow{d} &  \mathcal  H^1_{dR,c}(\Omega)\otimes \mathcal P_1 \arrow[r, ""{coordinate, name=Y}] \arrow{d} &  \mathcal H^2_{dR,c}(\Omega) \otimes\mathcal P_1 \arrow[r, ""{coordinate, name=X}] \arrow{d}&  \mathcal H^3_{dR,c}(\Omega) \otimes \mathcal P_1\ar[r] \ar[d]&  0\\
        0 \arrow{r}& V^0_0 \arrow{r} \arrow[uur, rounded corners, dashed, to path={ -- ([yshift=-4ex]\tikztostart.south)
        -| (Z) [near end]\tikztonodes
        |- ([yshift=4ex]\tikztotarget.north)
        -- (\tikztotarget)}]
        & 0 \arrow{r} \arrow[uur, rounded corners, dashed, to path={ -- ([yshift=-4ex]\tikztostart.south)
        -| (Y) [near end]\tikztonodes
        |- ([yshift=4ex]\tikztotarget.north)
        -- (\tikztotarget)}] &0  \arrow{r} \arrow[uur, rounded corners, dashed, to path={ -- ([yshift=-4ex]\tikztostart.south)
        -| (X) [near end]\tikztonodes
        |- ([yshift=4ex]\tikztotarget.north)
        -- (\tikztotarget)}] & 0 \ar[r] & 0.
         \end{tikzcd}
        \end{equation}
    The result follows from a diagram chase.
\end{proof}

\begin{proof}[Proof of \Cref{hom:divdiv-3d-trimmed}]
The proof is based on the following  dual pairs:
\begin{equation}
    \begin{tikzcd}[column sep = large]
   0 \arrow{r} & V^0_0 \arrow{r}{\hess_0} \arrow[d, phantom, "\ast"] &\bm V^1_0 \arrow{r}{\curl_0} \arrow[d, phantom, "\ast"] & \bm V^2_0 \arrow{r}{\div_0} \arrow[d, phantom, "\ast"]& \bm V^3_0 \arrow{r}{} \arrow[d, phantom, "\ast"] &0\\
    0 &U^3 \arrow{l}& \widehat{\bm U}^2 \arrow{l}[swap]{\widetilde{\div\div}} & \widehat{\bm U}^1 \arrow{l}[swap]{\widetilde{\sym\curl}}  & \bm U^0 \arrow{l}[swap]{\widetilde{\dev\grad}} &0 \arrow{l}{}.
     \end{tikzcd}
    \end{equation}
We can define $\pi_{U^{k} \to V^{3-k}_0}$ and the inner products of each space. It can be shown that these $\pi$'s induce an isomorphism between $\mathcal H^{k}(U^{\bs})$ and $\mathcal H^{3-k}(V^{\bs}_0)$, which concludes with the result.
\end{proof}

\begin{proof}[Proof of \Cref{thm:hom-divdiv0-3d} ]
Combining the dual pairing method
    \begin{equation}
        \begin{tikzcd}[column sep = large]
       0 \arrow{r} & V^0 \arrow{r}{\hess} \arrow[d, phantom, "\ast"] &\bm V^1 \arrow{r}{\curl} \arrow[d, phantom, "\ast"] & \bm V^2 \arrow{r}{\div} \arrow[d, phantom, "\ast"]& \bm V^3 \arrow{r}{} \arrow[d, phantom, "\ast"] &0\\
        0 &U^3_0 \arrow{l}& \widehat{\bm U}^2_0 \arrow{l}[swap]{\widehat{\div\div}_0} & \widehat{\bm U}^1_0 \arrow{l}[swap]{{\sym\curl}_h}  & \bm U^0_0 \arrow{l}[swap]{{\dev\grad}} &0 \arrow{l}{},
         \end{tikzcd}
        \end{equation}
        and the exactness of \begin{equation}
            \begin{tikzcd}[column sep=large]
            0 \ar[r] & 0 \ar[r] & \mathbb B_K^* \ar[r,"\sym\curl_h"] & \mathbb B_K \ar[r] & 0 \ar[r] & 0
            \end{tikzcd}
            \end{equation}
            yields the result.
\end{proof}

\section*{Acknowledgement}

The work of KH was supported by a Royal Society University Research Fellowship (URF$\backslash$ R1$\backslash$221398).

 \bibliographystyle{plain}
 \bibliography{ref}

\end{document}